\pdfpageattr{/Group <</S /Transparency /I true /CS /DeviceRGB>>}
\documentclass[letterpaper, twoside, 10pt, leqno]{amsart}

\usepackage{xcolor}
\usepackage{amssymb}
\usepackage{mathtools}
\usepackage{mathrsfs}
\usepackage{tikz}
    \usetikzlibrary{cd}
\usepackage{enumitem}
    \setlist[enumerate]{label={\textnormal{(\alph*)~\,}}, ref={(\alph*)}, nosep, align=left, left={\parindent}, labelsep=0pt, widest=b, rightmargin=0pt}
\usepackage[alphabetic, nobysame]{amsrefs}
\usepackage[colorlinks=true, allcolors=cyan!50!blue]{hyperref}
\usepackage[capitalise, nameinlink, noabbrev]{cleveref}
    
\usepackage[hyphenbreaks]{breakurl}

\allowdisplaybreaks
\numberwithin{equation}{section}

\theoremstyle{definition}
    \newtheorem{pdef}{Definition}[section]
\theoremstyle{plain}
    \newtheorem{pthrm}{Theorem}[section]
    \newtheorem{pmain}{Theorem}
    \newtheorem{pmaincor}[pmain]{Corollary}
    \newtheorem{plem}[pthrm]{Lemma}
    \newtheorem{pcor}[pthrm]{Corollary}
    \newtheorem{pprop}[pthrm]{Proposition}
\theoremstyle{remark}
    \newtheorem{prmrk}{Remark}[section]
    \newtheorem{pwarn}{Warning}[section]

\crefname{section}{Section}{Sections}
\crefname{subsection}{Subsection}{Subsections}
\crefname{pdef}{Definition}{Definitions}
\crefname{pthrm}{Theorem}{Theorems}
\crefname{pmain}{Theorem}{Theorems}
\crefname{pmaincor}{Corollary}{Corollaries}
\crefname{plem}{Lemma}{Lemmas}
\crefname{pcor}{Corollary}{Corollaries}
\crefname{pprop}{Proposition}{Propositions}
\crefname{prmrk}{Remark}{Remarks}
\crefname{pwarn}{Warning}{Warnings}
\AddToHook{env/pdef/begin}{\crefalias{pdef}{pdef}}
\AddToHook{env/pthrm/begin}{\crefalias{pthrm}{pthrm}}
\AddToHook{env/pmain/begin}{\crefalias{pmain}{pmain}}
\AddToHook{env/pmaincor/begin}{\crefalias{pmain}{pmaincor}}
\AddToHook{env/plem/begin}{\crefalias{pthrm}{plem}}
\AddToHook{env/pcor/begin}{\crefalias{pthrm}{pcor}}
\AddToHook{env/pprop/begin}{\crefalias{pthrm}{pprop}}
\AddToHook{env/prmrk/begin}{\crefalias{prmrk}{prmrk}}
\AddToHook{env/pwarn/begin}{\crefalias{pwarn}{pwarn}}

\newcommand{\injto}{\hookrightarrow}
\newcommand{\surjto}{\twoheadrightarrow}

\newcommand{\isoto}[1][\sim]{\stackrel{#1}{\to}}

\newcommand{\del}{\nabla}

\newcommand{\Lap}{\triangle}

\DeclareMathOperator{\Hess}{Hess}

\newcommand{\vol}{\mathrm{vol}}

\newcommand{\Ric}{\mathrm{Ric}}
\DeclareMathOperator{\tr}{tr}

\DeclareMathOperator{\Grad}{grad}
\DeclareMathOperator{\Div}{div}
\newcommand{\II}{\mathrm{I\!\!\:I}}

\begin{document}

\title[A Ruh-Vilms theorem in Weitzenb{\"o}ck geometry]{A Ruh-Vilms theorem for hypersurfaces \\ in Weitzenb{\"o}ck geometry}
\author[Dongha Lee]{Dongha Lee}
\address{School of Mathematics, Korea Institute for Advanced Study, 85, Hoe\-gi-ro, Dong\-dae\-mun-gu, Seoul, Republic of Korea}
\email{leejydh97@kias.re.kr}
\date{\today}
\begin{abstract}
A well-known theorem by Ruh and Vilms states that the Laplacian of the Gauss map for a smooth immersion into Euclidean space is given by the covariant derivative of the mean curvature vector field. For hypersurfaces, this implies that the Gauss map is harmonic iff the mean curvature is constant. In this paper, we extend this result to hypersurfaces in Weit\-zen\-b{\"o}ck geometry. While Riemannian geometry corresponds to the curved geometry without torsion, Weit\-zen\-b{\"o}ck geometry is a flat geometry with torsion. They represent two opposite extremes of Riemann-Cartan geometry.
\end{abstract}
\maketitle

\section{Introduction}

A celebrated work by Ruh and Vilms~\cite{MR0259768} establishes that the Laplacian of the Gauss map associated with a smooth immersion into the Euclidean space coincides with the covariant derivative of the mean curvature vector field. For a connected oriented smooth hypersurface \( M \) immersed in \( \mathbb{R}^{n+1} \), this result implies that the Gauss map \( g \colon M \to \mathbb{S}^n \) is harmonic iff \( M \) has constant mean curvature (CMC). This significantly facilitates the study of CMC hypersurfaces in Euclidean space and has yielded a wide range of applications. There has been extensive and continuous research to extend Ruh and Vilms' result to various geometric settings such as other Gauss maps and different ambient spaces. Notable contributions include~\cite{MR0859957} for the spherical Gauss map, \cites{MR0829396,MR0879397} for the Gauss section, \cite{MR0978614} for certain Gauss maps in a general Riemannian setting, \cite{MR2002821} for ambient Lie groups, \cite{MR1993542} for the ambient \( 3 \)-sphere, \cites{MR2231651,MR3456159} for ambient symmetric spaces, \cite{MR2343386} for the ambient \( \mathbb{H}^2 \times \mathbb{R} \), \cite{MR2763624} for the ambient anti-de Sitter \( 3 \)-space, \cite{MR4132000} for the ambient \( 7 \)-sphere, \cite{MR4159849} for the ambient complex projective plane, and~\cite{MR4795879} for the ambient para-complex projective plane.

The main result of the present paper is an extension of Ruh and Vilms' theorem to hypersurfaces in \emph{Weit\-zen\-b{\"o}ck geometry}, which is a flat geometry with torsion. Connections with torsion were pioneered and profoundly studied by Cartan~\cites{MR1509253,MR1509255,MR1509263}. His seminal work laid the foundation for \emph{Einstein-Cartan theory} in theoretical physics and \emph{Riemann-Cartan geometry} in mathematics. To be precise, a \emph{Riemann-Cartan manifold} is a smooth manifold endowed with a Riemannian metric and a metric-compatible connection, which may have nonzero \emph{torsion} and \emph{curvature}. A \emph{Riemannian manifold} with the Levi-Civita connection corresponds to the case with zero torsion, while a \emph{Weit\-zen\-b{\"o}ck manifold} lies at the other extreme with zero curvature. It is a parallelized manifold equipped with the natural flat Riemann-Cartan structure induced by its parallelization. For more details about these geometries and historical background, see Section~1 of~\cite{MR3468462}, Section~2 of~\cite{MR5011406}, Section~2 of~\cite{MR2037618}, and Section~1 of~\cite{MR2866744}.

The standard framework for submanifolds in Riemannian geometry is naturally extended to the Riemann-Cartan setting, despite the presence of nonzero torsion. This extension was systematically established in~\cite{MR5011406}, providing the necessary groundwork for the discussion in this paper. A crucial difference from the Riemannian setting is that the second fundamental form fails to be symmetric in general. Indeed, it is symmetric exactly when the differential \( 2 \)-form \( \tau \) defined by the normal component of the ambient torsion vanishes (see~\eqref{E:tau1} and~\eqref{E:tau2}). In the case of surfaces, it turned out that the Hodge dual of this \( 2 \)-form \( \tau \) produces a smooth scalar field \( {\star} \tau \) that serves as an ``imaginary counterpart'' to the mean curvature \( H \). It was observed that the complex-valued geometric quantity \( H + \boldsymbol{i} {\star} \tau \) plays a central role in extending various results in minimal surface theory.

One of the key features of submanifolds in Weit\-zen\-b{\"o}ck geometry is that the notion of the Gauss map arises naturally and serves as a central tool. By virtue of the existence of a fixed ambient global smooth frame, it is possible to define the Gauss map of a hypersurface as the tuple of the coefficient functions of the unit normal vector field. Then it is a map from the hypersurface to the unit sphere, which directly extends the classical Gauss map in Euclidean geometry. This map is the primary object of study in this paper, and the main interest lies in its Laplacian and harmonicity. Based on this map, generalizations of several classical results are presented in~\cite{MR5011406}, including the well-known theorem that characterizes minimal surfaces in \( \mathbb{R}^3 \) in terms of the conformality of the Gauss map.

The sustained interest in Ruh and Vilms' theorem has led to the aforementioned extensions in a wide range of geometric settings. The significance of the present work does not lie merely in providing another formal generalization. For surfaces, the complex-valued quantity \( H + \boldsymbol{i} {\star} \tau \) studied in~\cite{MR5011406} reappears naturally in the context of the Laplacian of the Gauss map (\cref{T:B,C:B,C:C}), thereby revealing another geometric aspect. The relevance of this quantity is also supported by its essential role in the renormalization of the Chern-Simons invariant for convex-cocompact hyperbolic \( 3 \)-manifolds~\cite{LeeCS}. Moreover, the special case of parallel totally skew-symmetric torsion (\cref{T:D}) recovers the earlier result of~\cite{MR2002821} for hypersurfaces in Lie groups with bi-invariant metrics (\cref{C:Lie-group}). Finally, as a substantial application, a Weit\-zen\-b{\"o}ck version of a theorem of Hoffman, Osserman, and Schoen~\cite{MR0694604} is established (\cref{T:HOS1,T:HOS2}), which contains the classical Euclidean case.

This paper is organized as follows. \cref{S:geometry-with-torsion} reviews Riemann-Cartan and Weit\-zen\-b{\"o}ck geometries, including their differential operators and hypersurfaces. The Riemann-Cartan version of the Codazzi equation together with its corollary is presented as a key ingredient (\cref{P:Codazzi-equation,C:mean-curvature}). Then \cref{S:Laplacian} is devoted to the Laplacian of the Gauss map. A Riemann-Cartan manifold has two connections, one of which is the Levi-Civita connection. For a Riemann-Cartan hypersurface \( M \), there are thus two possible Laplacians, \( \Lap g \) and \( \bar{\Lap} g \), of the Gauss map \( g \colon M \to \mathbb{S}^n \), with respect to the induced connection \( \del \) on \( M \) and the Levi-Civita connection \( \bar{\del} \) on \( M \), unless \( M \) is torsion-free. In this paper, both of them are covered. \cref{SS:hypersurfaces} starts with the general computation for hypersurfaces, yielding \cref{T:A}. The particular case of torsion-free hypersurfaces is the following.

\begin{pmain}[\cref{T:A,C:A}]
Let \( M \) be an oriented torsion-free smooth hypersurface embedded in a Weit\-zen\-b{\"o}ck\/ \( (n+1) \)-manifold. Let \( g \colon M \to \mathbb{S}^n \) be the Gauss map. Then, under the identification \( \phi \colon TM \isoto g^* T \mathbb{S}^n \) in \cref{L:identification},
\begin{align}
- \phi^{-1} \Lap g & = \Div (HI + 2W^{\mathsf{a}}), \label{E:A}
\end{align}
where \( H \) is the mean curvature, \( I \) is the identity operator, and \( W^{\mathsf{a}} \) is the skew-symmetric part of the Weingarten map. In particular, the Gauss map is harmonic iff the smooth\/ \( (1,1) \)-tensor field \( HI + 2W^{\mathsf{a}} \) is divergence-free.
\end{pmain}

\noindent This implies Ruh and Vilms' theorem for hypersurfaces in \( \mathbb{R}^{n+1} \) (\cref{C:RV}), as the right-hand side of~\eqref{E:A} becomes \( \Grad H \) for such hypersurfaces.

The focus narrows to surfaces in \cref{SS:surfaces}, and the result is the following.

\begin{pmain}[\cref{T:B,C:B}]
Let \( M \) be an oriented smooth surface embedded in a Weit\-zen\-b{\"o}ck\/ \( 3 \)-manifold. Let \( g \colon M \to \mathbb{S}^2 \) be the Gauss map. Then, under the identification \( \phi \colon TM \isoto g^* T \mathbb{S}^2 \) in \cref{L:identification},
\begin{align}
- \phi^{-1} \Lap g & = \Grad H + J \Grad {\star} \tau + (HI + ({\star} \tau) J - W) (\tr T)^{\sharp}, \\
- \phi^{-1} \bar{\Lap} g & = \Grad H + J \Grad {\star} \tau + (HI + ({\star} \tau) J) (\tr T)^{\sharp},
\end{align}
where \( H \) is the mean curvature, \( \tau \in \Omega^2(M,\mathbb{R}) \) is the normal component of the ambient torsion, \( I \) is the identity operator, \( J \) is the almost complex structure, \( W \) is the Weingarten map, and\/ \( \tr T \in \Omega^1(M,\mathbb{R}) \) is a contraction of the torsion of \( M \) (\cref{D:contraction}). In particular, if \( H + \boldsymbol{i} {\star} \tau = 0 \) everywhere, then the Gauss map is harmonic with respect to the Levi-Civita connection on \( M \).
\end{pmain}

\noindent For torsion-free surfaces, combining with a result of~\cite{MR5011406} yields the following equivalence, which generalizes a classical result for surfaces in \( \mathbb{R}^3 \) (\cref{C:Hopf}). See also Section~3 of~\cite{MR3456159} for related results.

\begin{pmaincor}[\cref{C:C}]
Let \( M \) be an oriented torsion-free smooth surface embedded in a Weit\-zen\-b{\"o}ck\/ \( 3 \)-manifold. Then the following three are equivalent.
\begin{enumerate}
\item The Gauss map \( g \colon M \to \mathbb{S}^2 \) is harmonic.
\item \( H + \boldsymbol{i} {\star} \tau \) is a holomorphic function on \( M \).
\item The Hopf differential \( \varphi \) on \( M \) is a holomorphic quadratic differential.
\end{enumerate}
\end{pmaincor}

The above results underscore the significance of torsion-free hypersurfaces. Accordingly, the focus of \cref{SS:torsion-free-hypersurfaces} is on their abundance and an explicit example. The investigation continues in \cref{SS:totally-skew-symmetric-torsion}, beginning with the study of ambient manifolds admitting only torsion-free hypersurfaces. Nontrivial examples of such ambient manifolds occur only in dimension three (\cref{P:torsion-free}), characterized by having \emph{totally skew-symmetric torsion} (\cref{D:totally-skew-symmetric}). Indeed, connections with totally skew-symmetric torsion have been of considerable interest in both physics and mathematics (e.g., \cites{MR1006380,MR1928632,MR2047649,MR3261964,MR4148347,MR4184296,MR4574239}). For parallel totally skew-symmetric ambient torsions, the following result is obtained, which recovers Theorem~1 of~\cite{MR2002821} about ambient Lie groups (\cref{C:Lie-group}).

\begin{pmain}[\cref{T:D}]
Let \( M \) be an oriented smooth hypersurface embedded in a Weit\-zen\-b{\"o}ck\/ \( (n+1) \)-manifold with parallel totally skew-symmetric torsion. Let \( g \colon M \to \mathbb{S}^n \) be the Gauss map. Then, under the identification \( \phi \colon TM \isoto g^* T \mathbb{S}^n \) in \cref{L:identification},\/ \( - \phi^{-1} \Lap g = \Grad H \), where \( H \) is the mean curvature. In particular, if \( M \) is connected, the Gauss map is harmonic iff \( M \) has constant mean curvature.
\end{pmain}

A well-known application of Ruh and Vilms' theorem is a classification due to Hoffman, Osserman, and Schoen~\cite{MR0694604}: A connected complete oriented CMC surface in \( \mathbb{R}^3 \) whose Gauss map has image contained in a hemisphere must be either a plane or a cylinder. Many previous studies extending Ruh and Vilms' work have also addressed corresponding analogues~\cites{MR2002821,MR2231651,MR3456159,MR4132000}. Following this line of work, we obtain two related results in \cref{SS:HOS-theorem} (\cref{T:HOS1,T:HOS2}), from which the classical Euclidean case follows (\cref{C:HOS}).

\section{Geometry with torsion} \label{S:geometry-with-torsion}

This section collects the preliminaries needed for the present paper. It starts with a brief review of Riemann-Cartan and Weit\-zen\-b{\"o}ck geometries in \cref{SS:Riemann-Cartan-geometry}. It then discusses their differential operators such as gradient, divergence, and Laplacian in \cref{SS:gradient-divergence-Laplacian} and the framework for their hypersurfaces in \cref{SS:Riemann-Cartan-hypersurfaces}. It should be noted that this paper focuses exclusively on the case of the Riemannian signature, whereas the literature on Riemann-Cartan geometry often admits pseudo-Riemannian signatures.

\subsection{Riemann-Cartan geometry} \label{SS:Riemann-Cartan-geometry}

A \emph{Riemann-Cartan manifold} is a Riemannian manifold endowed with a metric-compatible connection, possibly with nonzero torsion. A trivial example is a Riemannian manifold with the Levi-Civita connection. A Riemann-Cartan manifold also has the Levi-Civita connection as a Riemannian manifold. Quantities associated with the Levi-Civita connection will be denoted by a bar in the present paper. For example, \( \bar{\del} \), \( \bar{R} \), and \( \overline{\Ric} \).

For a Riemann-Cartan manifold \( M \) with metric \( \langle {{}\cdot{}}, {{}\cdot{}} \rangle \) and connection \( \del \), the following smooth tensor fields on \( M \) are well defined. The \emph{contorsion} \( C \) and the \emph{torsion} \( T \) of \( M \) are the smooth \( (1,2) \)-tensor fields on \( M \) defined by
\begin{align}
\del_X Y & = \bar{\del}_X Y + C(X,Y), \qquad T(X,Y) = \del_X Y - \del_Y X - [X,Y],
\end{align}
where \( X, Y \in \Gamma(TM) \). \( M \) is \emph{torsion-free} iff \( T = 0 \) everywhere. The \emph{curvature} \( R \) of \( M \) is the smooth \( (1,3) \)-tensor field on \( M \) defined by
\begin{align}
R(X,Y) Z & = \del_X \del_Y Z - \del_Y \del_X Z - \del_{[X,Y]} Z,
\end{align}
and the \emph{Ricci curvature} \( \Ric \) of \( M \) is its contraction \( \Ric(X,Y) = \tr (R({{}\cdot{}}, X) Y) \), where \( X, Y, Z \in \Gamma(TM) \). \( M \) is \emph{flat} iff \( R = 0 \) everywhere. They satisfy
\begin{align}
T(X,Y) & = C(X,Y) - C(Y,X), \\
0 & = \langle C(X,Y), Z \rangle + \langle Y, C(X,Z) \rangle, \\
\langle Z, C(X,Y) \rangle & = \tfrac{1}{2} (\langle Z, T(X,Y) \rangle + \langle X, T(Z,Y) \rangle + \langle Y, T(Z,X) \rangle), \label{E:contorsion-torsion} \\
R(X,Y) Z & = \begin{aligned}[t]
& \bar{R}(X,Y) Z + (\bar{\del}_X C)(Y,Z) - (\bar{\del}_Y C)(X,Z) \label{E:curvature-contorsion} \\
& + C(X,C(Y,Z)) - C(Y,C(X,Z))
\end{aligned}
\end{align}
for all \( X, Y, Z \in \Gamma(TM) \). Along with these, \cref{L:contraction} below will be used later.

\begin{pdef} \label{D:contraction}
For a smooth \( (1,2) \)-tensor field \( P \) on \( M \), we define \( \tr P \in \Omega^1(M,\mathbb{R}) \) by \( (\tr P)(X) = \tr (P(X, {{}\cdot{}})) \), where \( X \in \Gamma(TM) \).
\end{pdef}

\noindent For a smooth \( (1,1) \)-tensor field \( A \) on \( M \), recall that the \emph{adjoint} \( A^{\mathsf{t}} \) of \( A \) is the unique smooth \( (1,1) \)-tensor field satisfying \( \langle X, AY \rangle = \langle A^{\mathsf{t}} X, Y \rangle \) for all \( X, Y \in \Gamma(TM) \), which produces a decomposition \( A = A^{\mathsf{s}} + A^{\mathsf{a}} \) by the \emph{symmetric part} \( A^{\mathsf{s}} = \frac{1}{2} (A + A^{\mathsf{t}}) \) and the \emph{skew-symmetric part} \( A^{\mathsf{a}} = \frac{1}{2} (A - A^{\mathsf{t}}) \).

\begin{plem} \label{L:contraction}
Let \( A \) and \( B \) be smooth\/ \( (1,1) \)-tensor fields on \( M \).
\begin{enumerate}
\item \label{L:contraction:a} \( \del_X (\tr A) = \tr (\del_X A) \) and\/ \( (\del_X A)^{\mathsf{t}} = \del_X A^{\mathsf{t}} \) for all \( X \in \Gamma(TM) \).
\item \label{L:contraction:b} \( \tr_M C(A {{}\cdot{}}, B {{}\cdot{}}) = \frac{1}{2} \tr_M T(A {{}\cdot{}}, B {{}\cdot{}}) + \tr ((AB^{\mathsf{t}})^{\mathsf{s}} T)^{\sharp} \), so\/ \( \tr_M C = (\tr T)^{\sharp} \).
\end{enumerate}
\end{plem}
\begin{proof}
Let \( \omega \in \Omega^1(U,\mathfrak{o}(n)) \) be the connection \( 1 \)-form relative to a local orthonormal smooth frame \( (E_1, \dotsc, E_n) \) for an open set \( U \) in \( M \). For~(a), fix \( X \in \Gamma(TM) \). The difference \( \del_X (\tr A) - \tr (\del_X A) \) is equal to
\begin{align*}
& \del_X \sum_{i = 1}^n \langle E_i, AE_i \rangle - \sum_{i = 1}^n \langle E_i, (\del_X A) E_i \rangle = \sum_{i = 1}^n (\langle \del_X E_i, AE_i \rangle + \langle E_i, A \del_X E_i \rangle) \\
& \qquad = \sum_{i, j = 1}^n (\omega^j_i(X) A^j_i + \omega^j_i(X) A^i_j) = \sum_{i, j = 1}^n (\omega^i_j(X) + \omega^j_i(X)) A^i_j = 0,
\end{align*}
from which the former equality follows. Applying \( \del_X \) to both sides of the equality \( \langle Y, AZ \rangle = \langle A^{\mathsf{t}} Y, Z \rangle \) yields the latter equality, where \( Y, Z \in \Gamma(TM) \) are given. Now, for~(b), use~\eqref{E:contorsion-torsion}. The left-hand side \( \sum_{i, j = 1}^n \langle E_i, C(AE_j, BE_j) \rangle E_i \) is equal to
\begin{align*}
& \frac{1}{2} \sum_{i, j = 1}^n (\langle E_i, T(AE_j, BE_j) \rangle + \langle AE_j, T(E_i, BE_j) \rangle + \langle BE_j, T(E_i, AE_j) \rangle) E_i \\
& \qquad = \frac{1}{2} \tr_M T(A {{}\cdot{}}, B {{}\cdot{}}) + \frac{1}{2} \sum_{i, j, k, \ell = 1}^n (A^k_j B^{\ell}_j + A^{\ell}_j B^k_j) \langle E_k, T(E_i, E_{\ell}) \rangle E_i.
\end{align*}
The second term is \( \sum_{i, j = 1}^n \langle (AB^{\mathsf{t}})^{\mathsf{s}} E_j, T(E_i, E_j) \rangle E_i \), or simply \( \tr ((AB^{\mathsf{t}})^{\mathsf{s}} T)^{\sharp} \).
\end{proof}

There are two fundamental quantities in Riemann-Cartan geometry: the \emph{torsion} and the \emph{curvature}. Riemannian geometry is the special case with zero torsion. At the other extreme with zero curvature, there is \emph{Weit\-zen\-b{\"o}ck geometry}, arising from parallelization. If a smooth manifold is parallelizable and given a global smooth frame \( (E_1, \dotsc, E_n) \) (i.e., it is \emph{parallelized}), it naturally possesses an oriented flat Riemann-Cartan structure: The orientation is obvious, and the metric and the metric-compatible connection, called the \emph{Weit\-zen\-b{\"o}ck connection}, are determined by \( \langle E_i, E_j \rangle = \delta_{ij} \) and \( \del E_i = 0 \), where \( i, j \in \{ 1, \dotsc, n \} \). A \emph{Weit\-zen\-b{\"o}ck manifold} is essentially a parallelized manifold, but with an emphasis on being equipped with this oriented flat Riemann-Cartan structure.

\begin{pwarn}
A flat Riemann-Cartan manifold is \emph{locally} a Weit\-zen\-b{\"o}ck manifold, as every point admits a local orthonormal smooth frame whose Weit\-zen\-b{\"o}ck connection coincides with the flat connection on the manifold. For this reason, in some literature, a Weit\-zen\-b{\"o}ck manifold merely means a flat Riemann-Cartan manifold.
\end{pwarn}

\noindent On \( \mathbb{R}^n \), the same Riemann-Cartan structure arises whether one considers the Riemannian structure given by the standard inner product or the Weit\-zen\-b{\"o}ck structure given by the standard global frame. In this sense, Riemannian and Weit\-zen\-b{\"o}ck geometries are two independent directions for generalizing Euclidean geometry.

Every closed orientable smooth \( 3 \)-manifold is parallelizable~\cite{MR1509530}, thus admitting a Weit\-zen\-b{\"o}ck structure. Another typical example of a Weit\-zen\-b{\"o}ck manifold is a Lie group. Any basis for the Lie algebra of a Lie group induces a left-invariant global smooth frame via the left translation, defining a Weit\-zen\-b{\"o}ck structure on the Lie group. For example, the Lie group \( \mathbb{S}^3 \) of unit quaternions admits the Weit\-zen\-b{\"o}ck structure given by the basis \( (\boldsymbol{i}, \boldsymbol{j}, \boldsymbol{k}) \).

\subsection{Gradient, divergence, and Laplacian} \label{SS:gradient-divergence-Laplacian}

Let \( M \) be a Riemann-Cartan manifold. The \emph{gradient} of a smooth function \( f \colon M \to \mathbb{R} \) is the smooth vector field \( \Grad f = (\del f)^{\sharp} = (\mathrm{d} f)^{\sharp} \) on \( M \). It does not depend on the connection on \( M \). The \emph{divergence} of a smooth \( (1,1) \)-tensor field \( A \) on \( M \) is the smooth vector field \( \Div A = \tr_M (\del A) \) on \( M \). It does depend on the connection on \( M \), and the difference from the Levi-Civita one is given as follows.

\begin{plem} \label{L:gradient-divergence}
Let \( M \) be a Riemann-Cartan manifold. Let \( f \colon M \to \mathbb{R} \) be a smooth function, and let \( A \) be a smooth\/ \( (1,1) \)-tensor field on \( M \).
\begin{enumerate}
\item \label{L:gradient-divergence:a} The Leibniz rule holds:\/ \( \Div(fA) = A \Grad f + f \Div A \).
\item \label{L:gradient-divergence:b} The difference between the two divergences of \( A \) is given by
\begin{align}
\begin{aligned}
\Div A - \overline{\Div}\, A & = \tr_M C({{}\cdot{}}, A {{}\cdot{}}) - A \tr_M C \\
& = \tfrac{1}{2} \tr_M T({{}\cdot{}}, A {{}\cdot{}}) + \tr (A^{\mathsf{s}} T)^{\sharp} - A (\tr T)^{\sharp}.
\end{aligned}
\end{align}
\end{enumerate}
\end{plem}
\begin{proof}
For~(a), note that \( \del (fA) = \mathrm{d} f \otimes A + f \del A \) and \( \tr_M (\mathrm{d} f \otimes A) = A (\mathrm{d} f)^{\sharp} \). For~(b), let \( (E_1, \dotsc, E_n) \) be a local orthonormal smooth frame for \( M \). Summing both sides of \( (\del_{E_i} A - \bar{\del}_{E_i} A) E_i = C(E_i, AE_i) - A C(E_i, E_i) \) over \( i \in \{ 1, \dotsc, n \} \) yields the first equality. The second equality follows from \cref{L:contraction}\labelcref{L:contraction:b}.
\end{proof}

Now, consider a smooth map \( F \colon M \to N \) between Riemann-Cartan manifolds. Its differential at \( p \in M \) is a linear map \( \mathrm{d} F |_p \colon T_p M \to T_{F(p)} N \), which defines a global smooth section \( \mathrm{d} F \in \Gamma(T^* M \otimes F^* TN) \). The connection on \( M \) induces a connection on the cotangent bundle \( T^* M \surjto M \), and the connection on \( N \) induces a connection on the pullback bundle \( F^* TN \surjto M \). Then they induce a connection on \( T^* M \otimes F^* TN \surjto M \), which yields the map \( \del (\mathrm{d} F) \colon \Gamma(TM) \to \Gamma(T^* M \otimes F^* TN) \). The \emph{Hessian} of \( F \) at \( p \in M \) is the bilinear map
\begin{align}
& \Hess F |_p = \del (\mathrm{d} F) |_p \colon T_p M \times T_p M \to T_{F(p)} N,
\end{align}
and the \emph{Laplacian} or \emph{tension field} of \( F \) is its trace \( \Lap F = \tr_M (\Hess F) \in \Gamma(F^* TN) \) with respect to the metric on \( M \). \( F \) is \emph{harmonic} iff \( \Lap F = 0 \) everywhere. If \( (x^i) \) and \( (y^{\alpha}) \) are local charts near \( p \in M \) and \( F(p) \in N \) respectively,
\begin{align}
(\Hess F)^{\alpha}_{ij} & = \frac{\partial^2 F^{\alpha}}{\partial x^i \partial x^j} - \frac{\partial F^{\alpha}}{\partial x^k} (\Gamma_M)^k_{ij} + \frac{\partial F^{\beta}}{\partial x^i} \frac{\partial F^{\gamma}}{\partial x^j} (\Gamma_N)^{\alpha}_{\beta \gamma} \circ F
\end{align}
under the Einstein summation convention, where \( \Gamma_M \) and \( \Gamma_N \) are the Christoffel symbols of the connections on \( M \) and \( N \) respectively.

As a Riemann-Cartan manifold possesses two connections, there are four possible Hessians of a smooth map \( F \colon M \to N \) between Riemann-Cartan manifolds:
\begin{align}
& \Hess_{\del^M, \del^N} F, \qquad \Hess_{\bar{\del}^M, \del^N} F, \qquad \Hess_{\del^M, \bar{\del}^N} F, \qquad \Hess_{\bar{\del}^M, \bar{\del}^N} F.
\end{align}
Accordingly, there are four possible corresponding Laplacians. The difference between the first two is given as follows.

\begin{pprop} \label{P:Hessian-Laplacian}
Let \( F \colon M \to N \) be a smooth map between Riemann-Cartan manifolds. Then
\begin{align}
\Hess_{\del^M, \del^N} F - \Hess_{\bar{\del}^M, \del^N} F & = - \mathrm{d} F \circ C^M, \\
\Lap_{\del^M, \del^N} F - \Lap_{\bar{\del}^M, \del^N} F & = - \mathrm{d} F (\tr T_M)^{\sharp},
\end{align}
where \( C^M \) and \( T_M \) are the contorsion and the torsion of \( M \) respectively. Note that the right-hand sides do not depend on\/ \( \del^N \).
\end{pprop}
\begin{proof}
For the former equality, observe that
\begin{align*}
& {\left( (F^* \del^N)_X (\mathrm{d} F (Y)) - \mathrm{d} F (\del^M_X Y) \right)} - {\left( (F^* \del^N)_X (\mathrm{d} F (Y)) - \mathrm{d} F (\bar{\del}^M_X Y) \right)} \\
& \qquad = - \mathrm{d} F (\del^M_X Y - \bar{\del}^M_X Y) = - \mathrm{d} F (C^M(X,Y))
\end{align*}
for all \( X, Y \in \Gamma(TM) \), where \( F^* \del^N \) is the pullback connection on \( F^* TN \surjto M \). For the latter equality, recall from \cref{L:contraction}\labelcref{L:contraction:b} that \( \tr_M C^M = (\tr T_M)^{\sharp} \).
\end{proof}

In particular, consider a smooth map \( F \colon M \to \mathbb{S}^n \), where \( M \) is a Riemann-Cartan manifold. Let \( \hat{\iota} \colon \mathbb{S}^n \injto \mathbb{R}^{n+1} \) be the inclusion, and let \( \hat{F} = \hat{\iota} \circ F \colon M \to \mathbb{R}^{n+1} \). It is elementary that \( \del^{\mathbb{R}^{n+1}}_{\mathrm{d} \hat{\iota} (V)} (\mathrm{d} \hat{\iota} (W)) = \mathrm{d} \hat{\iota} \, \del^{\mathbb{S}^n}_V W - \langle V, W \rangle \nu \) for all \( V, W \in \Gamma(T \mathbb{S}^n) \), where \( \nu \) is the outward unit normal vector field along \( \mathbb{S}^n \). It follows that
\begin{align}
(\hat{F}^* \del^{\mathbb{R}^{n+1}})_X (\mathrm{d} \hat{F} (Y)) & = \mathrm{d} \hat{\iota} (F^* \del^{\mathbb{S}^n})_X (\mathrm{d} F (Y)) - \langle \mathrm{d} F (X), \mathrm{d} F (Y) \rangle \nu, \\
(\del (\mathrm{d} \hat{F}))_X Y & = \mathrm{d} \hat{\iota} \, (\del (\mathrm{d} F))_X Y - \langle \mathrm{d} F (X), \mathrm{d} F (Y) \rangle \nu
\end{align}
for all \( X, Y \in \Gamma(TM) \). Therefore, the Laplacians satisfy
\begin{align}
\Lap \hat{F} & = \mathrm{d} \hat{\iota} \, \Lap F - \tr_M \langle \mathrm{d} F, \mathrm{d} F \rangle \nu = \mathrm{d} \hat{\iota} \, \Lap F - \| \mathrm{d} F \|^2 \nu \label{E:Laplacian}
\end{align}
as global smooth sections of \( \hat{F}^* T \mathbb{R}^{n+1} \surjto M \).

\subsection{Riemann-Cartan hypersurfaces} \label{SS:Riemann-Cartan-hypersurfaces}

Even in the presence of nonzero torsion, the standard framework for submanifolds in Riemannian geometry naturally extends to the Riemann-Cartan setting as in~\cite{MR5011406}. For our purposes, it is sufficient to recall the codimension-one case. Let \( M \) be a smooth hypersurface embedded in a Riemann-Cartan manifold \( \tilde{M} \). Then the metric on \( \tilde{M} \) restricts to a metric on \( M \), and the metric-compatible connection on \( \tilde{M} \) projects to a metric-compatible connection on \( M \). In this manner, \( M \) inherits a Riemann-Cartan structure, being a \emph{Riemann-Cartan hypersurface}. However, even if \( \tilde{M} \) is a Weit\-zen\-b{\"o}ck manifold, the inherited Riemann-Cartan structure on \( M \) is not necessarily Weit\-zen\-b{\"o}ck (e.g., spheres in the Euclidean space). In the present paper, quantities associated with the ambient manifold will be denoted by a tilde. For example, \( \tilde{\del} \), \( \tilde{T} \), \( \tilde{R} \), and \( \widetilde{\Ric} \).

Suppose further that \( M \) and \( \tilde{M} \) are oriented. There is a unique unit normal vector field \( N \) along \( M \) such that a frame \( (E_1, \dotsc, E_n) \) for \( M \) is oriented iff the frame \( (N, E_1, \dotsc, E_n) \) for \( \tilde{M} \) is oriented. The \emph{second fundamental form} \( \II \) of \( M \) is the smooth \( (0,2) \)-tensor field on \( M \) defined by \( \II(X,Y) = \langle N, \tilde{\del}_X Y \rangle \), where \( X, Y \in \Gamma(TM) \) are extended arbitrarily smoothly. \( M \) is \emph{totally geodesic} iff \( \II = 0 \) everywhere. The \emph{Weingarten map} \( W \) of \( M \) is the unique smooth \( (1,1) \)-tensor field on \( M \) satisfying \( \II(X,Y) = \langle WX, Y \rangle \) for all \( X, Y \in \Gamma(TM) \). Then \( W = - \tilde{\del} N \) if \( N \) is extended arbitrarily smoothly. The \emph{extrinsic Gaussian curvature} and the \emph{mean curvature} of \( M \) are defined by \( K^{\mathrm{e}} = \det W \) and \( H = \tr W \) respectively. A crucial difference from Riemannian geometry is that the second fundamental form is not necessarily symmetric: For all \( X, Y \in \Gamma(TM) \),
\begin{align}
\II(X,Y) - \II(Y,X) & = \langle N, \tilde{\del}_X Y - \tilde{\del}_Y X \rangle = \langle N, \tilde{T}(X,Y) \rangle = \tau(X,Y), \label{E:tau1} \\
\langle 2W^{\mathsf{a}} X, Y \rangle & = \langle WX, Y \rangle - \langle X, WY \rangle = \tau(X,Y), \label{E:tau2}
\end{align}
where \( \tau = \langle N, \tilde{T} \rangle \in \Omega^2(M,\mathbb{R}) \) (Definition~3.3 of~\cite{MR5011406}). It is symmetric exactly when the differential \( 2 \)-form \( \tau \) vanishes.

Here we state and prove the Riemann-Cartan version of the Codazzi equation and derive a corollary that will be a key ingredient later.

\begin{pprop}[Codazzi equation, Riemann-Cartan version] \label{P:Codazzi-equation}
Let \( M \) be a smooth hypersurface embedded in an oriented Riemann-Cartan manifold \( \tilde{M} \), oriented by a unit normal vector field \( N \) along \( M \). Then, for all \( X, Y \in \Gamma(TM) \),
\begin{align}
(\del_X W) Y - (\del_Y W) X & = - \tilde{R}(X,Y) N - WT(X,Y).
\end{align}
\end{pprop}
\begin{proof}
Fix \( X, Y \in \Gamma(TM) \). The left-hand side is equal to
\begin{align*}
& (\del_X (WY) - \del_Y (WX)) - W (\del_X Y - \del_Y X) \\
& \qquad = - (\del_X \tilde{\del}_Y N - \del_Y \tilde{\del}_X N) - WT(X,Y) + \tilde{\del}_{[X,Y]} N \\
& \qquad = - (\tilde{R}(X,Y) N)^{\top} - WT(X,Y).
\end{align*}
Since \( \langle \tilde{R}(X,Y) N, N \rangle = 0 \) by Proposition~2.1 of~\cite{MR5011406}, \( \tilde{R}(X,Y) N \in \Gamma(TM) \).
\end{proof}

\begin{pcor} \label{C:mean-curvature}
In the setting of \cref{P:Codazzi-equation},
\begin{align}
\mathrm{d} H & = (\Div W^{\mathsf{t}})^{\flat} + \widetilde{\Ric}({{}\cdot{}}, N) - \tr(WT).
\end{align}
\end{pcor}
\begin{proof}
Let \( (E_1, \dotsc, E_n) \) be a local oriented orthonormal smooth frame for \( M \). Fix \( X \in \Gamma(TM) \). By \cref{L:contraction}\labelcref{L:contraction:a} and \cref{P:Codazzi-equation}, \( \mathrm{d} H (X) = \tr (\del_X W) \), and
\begin{align*}
\langle E_i, (\del_X W) E_i \rangle & = \langle (\del_{E_i} W^{\mathsf{t}}) E_i, X \rangle - \langle E_i, \tilde{R}(X,E_i) N \rangle - \langle E_i, W T(X,E_i) \rangle
\end{align*}
for all \( i \in \{ 1, \dotsc, n \} \). Note that \( \langle N, \tilde{R}(X,N) N \rangle = 0 \) by Proposition~2.1 of~\cite{MR5011406}. Summing over \( i \in \{ 1, \dotsc, n \} \) completes the proof.
\end{proof}

The Weit\-zen\-b{\"o}ck case has a powerful tool: the Gauss map. Let \( M \) be an oriented smooth hypersurface embedded in a Weit\-zen\-b{\"o}ck \( (n+1) \)-manifold \( \tilde{M} \) with global smooth frame \( \tilde{s} = (\tilde{E}_1, \dotsc, \tilde{E}_{n+1}) \), where \( M \) is oriented by the unit normal vector field \( N = N^1 \tilde{E}_1 + \dotsb + N^{n+1} \tilde{E}_{n+1} \) along \( M \). The \emph{Gauss map} of \( M \) is defined by the smooth map \( g = (N^1, \dotsc, N^{n+1}) \colon M \to \mathbb{S}^n \). As in Euclidean geometry, the Weingarten map is expressed in terms of the Gauss map: \( W = - \sum_{i = 1}^{n+1} \mathrm{d} N^i \otimes \tilde{E}_i \) (Proposition~3.12 of~\cite{MR5011406}). It depends only on the restriction \( \tilde{s} |_M \), not the whole \( \tilde{s} \). Also, it follows that \( M \) is totally geodesic iff the Gauss map is locally constant.

Lastly, consider surfaces. Let \( M \) be an oriented smooth surface embedded in an oriented Riemann-Cartan \( 3 \)-manifold. Let \( J \) be the almost complex structure on \( M \), and let \( \varphi = \II(\partial_z, \partial_z) \, \mathrm{d} z^2 \) be the \emph{Hopf differential} on \( M \), which is a quadratic differential. The Hodge dual of \( \tau \in \Omega^2(M,\mathbb{R}) \) produces a smooth scalar field \( {\star} \tau \) on \( M \). In~\cite{MR5011406}, it turned out that it plays the role of an ``imaginary counterpart'' to the mean curvature \( H \), so that the following definitions are introduced.
\begin{enumerate}
\item \( M \) is \emph{minimal} iff \( H + \boldsymbol{i} {\star} \tau = 0 \) everywhere.
\item \( M \) is \emph{totally umbilic} iff \( W = \frac{1}{2} (HI + ({\star} \tau) J) \), where \( I \) is the identity operator.
\end{enumerate}
These extend the usual definitions in Riemannian geometry. Since \( ({\star} \tau) J = 2W^{\mathsf{a}} \) by~\eqref{E:tau2}, they can be reformulated without \( {\star} \tau \), using ``\( HI + 2W^{\mathsf{a}} = 0 \)'' for~(a) and ``\( W^{\mathsf{s}} = \frac{1}{2} HI \)'' for~(b), which suggest definitions that make sense in higher dimensions as well. Notably, \( HI + 2W^{\mathsf{a}} \) will appear in \cref{T:A,C:A}.

It was observed that the complex-valued smooth scalar field \( H + \boldsymbol{i} {\star} \tau \) generalizes well-known fundamental results in minimal surface theory as follows.

\begin{pthrm}[Theorem~B of~\cite{MR5011406}] \label{T:Hopf-differential}
Let \( M \) be a smooth surface embedded in an oriented Riemann-Cartan\/ \( 3 \)-manifold, oriented by a unit normal vector field \( N \) along \( M \). Suppose that \( \tilde{R}(X,Y) N = JWJT(X,Y) \) for all \( X, Y \in \Gamma(TM) \). Then \( H + \boldsymbol{i} {\star} \tau \) is holomorphic on \( M \) iff the Hopf differential \( \varphi \) on \( M \) is holomorphic.
\end{pthrm}

\begin{pthrm}[Theorem~D of~\cite{MR5011406}]
Let \( M \) be a connected oriented smooth surface embedded in a Weit\-zen\-b{\"o}ck\/ \( 3 \)-manifold. Then the Gauss map \( g \colon M \to \mathbb{S}^2 \) is conformal (i.e., angle-preserving) iff \( M \) is nowhere geodesic (i.e.,\/ \( \II \ne 0 \) everywhere) and either it is minimal or totally umbilic.
\end{pthrm}

\noindent These extend the following classical facts. For a connected oriented smooth surface \( M \) embedded in an oriented Riemannian \( 3 \)-manifold \( \tilde{M} \) of constant sectional curvature, \( M \) has constant mean curvature iff the Hopf differential \( \varphi \) on \( M \) is holomorphic. Also, if \( \tilde{M} = \mathbb{R}^3 \), the Gauss map \( g \colon M \to \mathbb{S}^2 \) is conformal iff \( M \) is nowhere geodesic and either it is minimal or totally umbilic.

\section{Laplacian of the Gauss map} \label{S:Laplacian}

The main concern of this section is the Laplacian of the Gauss map of an oriented smooth hypersurface embedded in a Weit\-zen\-b{\"o}ck manifold. \cref{SS:hypersurfaces} carries out the general computation for hypersurfaces, while \cref{SS:surfaces} focuses on surfaces. As the results highlight the significance of torsion-free hypersurfaces, \cref{SS:torsion-free-hypersurfaces} discusses their existence. Then \cref{SS:totally-skew-symmetric-torsion} addresses a condition on ambient manifolds. Lastly, \cref{SS:HOS-theorem} considers an application.

\subsection{Hypersurfaces} \label{SS:hypersurfaces}

Throughout this subsection, consider an oriented smooth hypersurface \( M \) embedded in a Weit\-zen\-b{\"o}ck \( (n+1) \)-manifold \( \tilde{M} \) with global smooth frame \( \tilde{s} = (\tilde{E}_1, \dotsc, \tilde{E}_{n+1}) \), where \( M \) is oriented by the unit normal vector field \( N = N^1 \tilde{E}_1 + \dotsb + N^{n+1} \tilde{E}_{n+1} \) along \( M \). The goal of this subsection is to compute the Laplacian of the Gauss map \( g = (N^1, \dotsc, N^{n+1}) \colon M \to \mathbb{S}^n \), and the result is \cref{T:A}. At each \( p \in M \), there is an orientation-preserving linear isometry \( \phi \colon T_p \tilde{M} \to \mathbb{R}^{n+1} \) sending each vector to its coefficient vector relative to the basis \( \tilde{s} |_p \), which satisfies \( \phi(N |_p) = g(p) \) and restricts to an orientation-preserving linear isometry \( \phi \colon T_p M \to T_{g(p)} \mathbb{S}^n \). This induces the following identification.

\begin{plem} \label{L:identification}
The map \( \phi \) above induces an isomorphism \( \phi \colon TM \isoto g^* T \mathbb{S}^n \) between smooth real vector bundles over \( M \) of rank \( n \). Moreover, it preserves the connection:\/ \( (g^* \del^{\mathbb{S}^n})_X (\phi(Y)) = \phi(\del_X Y) \) for all \( X, Y \in \Gamma(TM) \).
\end{plem}
\begin{proof}
The nontrivial part is the preservation of the connection. Fix \( X \in \Gamma(TM) \). Instead of \( Y \in \Gamma(TM) \), it is equivalent to show that \( (g^* \del^{\mathbb{S}^n})_X \xi = \phi(\del_X (\phi^{-1} \xi)) \) for all \( \xi \in \Gamma(g^* T \mathbb{S}^n) \). By linearity, it suffices to check the case \( \xi = g^* V = V \circ g \) for some \( V \in \Gamma(T \mathbb{S}^n) \). Let \( V = \sum_{i = 1}^{n+1} V^i e_i \in \Gamma(T \mathbb{S}^n) \), where \( V^1, \dotsc, V^{n+1} \colon \mathbb{S}^n \to \mathbb{R} \) are smooth coefficient functions and \( (e_1, \dotsc, e_{n+1}) \) is the standard basis for \( \mathbb{R}^{n+1} \). Then the two sides of \( (g^* \del^{\mathbb{S}^n})_X \xi = \phi(\del_X (\phi^{-1} \xi)) \) respectively reduce to
\begin{gather*}
\del^{\mathbb{S}^n}_{g_* X} V = (\del^{\mathbb{R}^{n+1}}_{g_* X} V)^{\top} = \sum_{i = 1}^{n+1} (g_* X)(V^i) e_i^{\top} = \sum_{i = 1}^{n+1} X(V^i \circ g) e_i^{\top}, \\
\sum_{i = 1}^{n+1} \phi(\del_X ((V^i \circ g) \tilde{E}_i)) = \sum_{i = 1}^{n+1} \phi(\tilde{\del}_X ((V^i \circ g) \tilde{E}_i))^{\top} = \sum_{i = 1}^{n+1} \phi(X(V^i \circ g) \tilde{E}_i^{\top}).
\end{gather*}
They coincide.
\end{proof}
\begin{plem} \label{L:identification-connection}
Let \( A \) be a smooth\/ \( (1,1) \)-tensor field on \( M \). Then\/ \( \del_X (\phi A) = \phi \del_X A \) for all \( X \in \Gamma(TM) \).
\end{plem}
\begin{proof}
This follows from \cref{L:identification} by a straightforward computation.
\end{proof}

The Weingarten map \( W \) is given by the differential of the Gauss map \( g \colon M \to \mathbb{S}^n \) under the identification \( \phi \) in \cref{L:identification} as follows.

\begin{pprop} \label{P:identification}
We have \( \phi W = - \mathrm{d} g \).
\end{pprop}
\begin{proof}
Let \( (e_1, \dotsc, e_{n+1}) \) be the standard basis for \( \mathbb{R}^{n+1} \). Let \( \iota \colon M \injto \tilde{M} \) and \( \hat{\iota} \colon \mathbb{S}^n \injto \mathbb{R}^{n+1} \) be the inclusions. Fix \( p \in M \). Then \( \phi \circ \mathrm{d} \iota |_p = \mathrm{d} \hat{\iota} |_{g(p)} \circ \phi \). Therefore,
\begin{align*}
\mathrm{d} \hat{\iota} |_{g(p)} \circ \phi W & = \phi \circ \mathrm{d} \iota |_p W = - \phi \sum_{i = 1}^{n+1} \mathrm{d} N^i |_p \otimes \tilde{E}_i |_p = - \sum_{i = 1}^{n+1} \mathrm{d} N^i |_p \otimes e_i, \\
- \mathrm{d} \hat{\iota} |_{g(p)} \circ \mathrm{d} g |_p & = - \mathrm{d} (\hat{\iota} \circ g) |_p = - \mathrm{d} (N^1, \dotsc, N^{n+1}) |_p = - \sum_{i = 1}^{n+1} \mathrm{d} N^i |_p \otimes e_i.
\end{align*}
As \( \mathrm{d} \hat{\iota} |_{g(p)} \) is injective, this completes the proof.
\end{proof}

\begin{pthrm} \label{T:Laplacian}
Let \( \hat{g} = \hat{\iota} \circ g \colon M \to \mathbb{R}^{n+1} \), where \( \hat{\iota} \colon \mathbb{S}^n \injto \mathbb{R}^{n+1} \) is the inclusion. Then\/ \( - \phi^{-1} \Lap g = \Div W \) and\/ \( - \phi^{-1} \Lap \hat{g} = \Div W + \| W \|^2 N \).
\end{pthrm}
\begin{proof}
By \cref{P:identification,L:identification-connection}, for all \( X, Y \in \Gamma(TM) \),
\begin{align}
- \phi^{-1} (\del_X (\mathrm{d} g)) Y & = \phi^{-1} (\del_X (\phi W)) Y = \phi^{-1} (\phi \del_X W) Y = (\del_X W) Y.
\end{align}
The former equality follows. For the latter equality, recall~\eqref{E:Laplacian}.
\end{proof}

We are ready to present the main result of this subsection. Recall that there are two possible Laplacians of a smooth map defined on \( M \) and two possible divergences of a smooth \( (1,1) \)-tensor field on \( M \), depending on the connection on \( M \) (i.e., the induced one or the Levi-Civita one), unless \( M \) is torsion-free.

\begin{pthrm} \label{T:A}
Let \( M \) be an oriented smooth hypersurface embedded in a Weit\-zen\-b{\"o}ck\/ \( (n+1) \)-manifold. Let \( g \colon M \to \mathbb{S}^n \) be the Gauss map. Let \( \phi \colon TM \isoto g^* T \mathbb{S}^n \) be the identification in \cref{L:identification}. Then
\begin{align}
- \phi^{-1} \Lap g & = \Div (HI + 2W^{\mathsf{a}}) + \tr(WT)^{\sharp}, \\
- \phi^{-1} \Lap g & = \overline{\Div} (HI + 2W^{\mathsf{a}}) + \tr(WT)^{\sharp} - 2W^{\mathsf{a}} (\tr T)^{\sharp} + \tr_M T({{}\cdot{}}, W^{\mathsf{a}} {{}\cdot{}}), \\
- \phi^{-1} \bar{\Lap} g & = \Div (HI + 2W^{\mathsf{a}}) + \tr(WT)^{\sharp} + W (\tr T)^{\sharp}, \\
- \phi^{-1} \bar{\Lap} g & = \overline{\Div} (HI + 2W^{\mathsf{a}}) + \tr(WT)^{\sharp} + W^{\mathsf{t}} (\tr T)^{\sharp} + \tr_M T({{}\cdot{}}, W^{\mathsf{a}} {{}\cdot{}}),
\end{align}
where a bar denotes a quantity associated with the Levi-Civita connection on \( M \). In particular, if \( M \) is torsion-free,\/ \( - \phi^{-1} \Lap g = \Div (HI + 2W^{\mathsf{a}}) \).
\end{pthrm}
\begin{proof}
By \cref{L:gradient-divergence}\labelcref{L:gradient-divergence:a} and \cref{C:mean-curvature},
\begin{align}
\Div(HI) & = \Grad H = \Div W^{\mathsf{t}} - \tr(WT)^{\sharp} = \Div W - \Div(2W^{\mathsf{a}}) - \tr(WT)^{\sharp}.
\end{align}
By \cref{T:Laplacian}, the first equality follows. By \cref{L:gradient-divergence}\labelcref{L:gradient-divergence:b} and \cref{P:Hessian-Laplacian},
\begin{align}
\Div (HI + 2W^{\mathsf{a}}) & = \overline{\Div} (HI + 2W^{\mathsf{a}}) + \tr_M T({{}\cdot{}}, W^{\mathsf{a}} {{}\cdot{}}) - 2W^{\mathsf{a}} (\tr T)^{\sharp}, \\
- \phi^{-1} \Lap g & = - \phi^{-1} \bar{\Lap} g + \phi^{-1} \mathrm{d} g (\tr T)^{\sharp} = - \phi^{-1} \bar{\Lap} g - W (\tr T)^{\sharp},
\end{align}
where \( \phi^{-1} \mathrm{d} g = - W \) by \cref{P:identification}. The other equalities follow.
\end{proof}

\begin{prmrk}
The Laplacians \( \Lap \hat{g} \) and \( \bar{\Lap} \hat{g} \) of the smooth map \( \hat{g} = \hat{\iota} \circ g \colon M \to \mathbb{R}^{n+1} \) can also be obtained by \cref{T:A}, where \( \hat{\iota} \colon \mathbb{S}^n \injto \mathbb{R}^{n+1} \) is the inclusion, since \( - \phi^{-1} \Lap \hat{g} = - \phi^{-1} \Lap g + \| W \|^2 N \) and \( - \phi^{-1} \bar{\Lap} \hat{g} = - \phi^{-1} \bar{\Lap} g + \| W \|^2 N \) by~\eqref{E:Laplacian}.
\end{prmrk}

\begin{pcor} \label{C:A}
Let \( M \) be an oriented torsion-free smooth hypersurface embedded in a Weit\-zen\-b{\"o}ck\/ \( (n+1) \)-manifold. Then the Gauss map \( g \colon M \to \mathbb{S}^n \) is harmonic iff the smooth\/ \( (1,1) \)-tensor field \( HI + 2W^{\mathsf{a}} \) on \( M \) is divergence-free.
\end{pcor}

\begin{pcor}[\cite{MR0259768}] \label{C:RV}
Let \( M \) be a connected oriented smooth hypersurface embedded in\/ \( \mathbb{R}^{n+1} \). Let \( g \colon M \to \mathbb{S}^n \) be the Gauss map. Then\/ \( - \Lap g = \Grad H \). In particular, the Gauss map is harmonic iff \( M \) has constant mean curvature.
\end{pcor}

\begin{prmrk} \label{R:S}
The Laplacian of the Gauss map in \cref{T:A} consists of two parts:
\begin{align}
- \phi^{-1} \bar{\Lap} g & = \underbrace{\overline{\Div} (HI + 2W^{\mathsf{a}})}_{\textnormal{differential part}} + \underbrace{\tr(WT)^{\sharp} + W^{\mathsf{t}} (\tr T)^{\sharp} + \tr_M T({{}\cdot{}}, W^{\mathsf{a}} {{}\cdot{}})}_{\textnormal{algebraic part}}.
\end{align}
Here, the algebraic part can be further simplified by introducing a skew-symmetric smooth \( (1,2) \)-tensor field \( S \) on \( M \) defined by \( S(X,Y) = T(WX,Y) + T(X,WY) \), where \( X, Y \in \Gamma(TM) \). Indeed, \( \tr(WT)^{\sharp} + W^{\mathsf{t}} (\tr T)^{\sharp} = (\tr S)^{\sharp} \), since
\begin{align*}
& \sum_{i, j = 1}^n (\langle W^{\mathsf{t}} E_i, T(E_j, E_i) \rangle E_j + \langle E_i, T(E_j, E_i) \rangle W^{\mathsf{t}} E_j) \\
& \qquad = \sum_{i, j = 1}^n (\langle E_i, T(E_j, WE_i) \rangle E_j + \langle E_i, T(WE_j, E_i) \rangle E_j),
\end{align*}
where \( (E_1, \dotsc, E_n) \) is a local oriented orthonormal smooth frame for \( M \).
\end{prmrk}

\subsection{Surfaces} \label{SS:surfaces}

This subsection focuses on the case \( n = 2 \) of \cref{SS:hypersurfaces}. The main result is \cref{T:B}. In dimension two, recall from \cref{SS:Riemann-Cartan-hypersurfaces} that \( 2W^{\mathsf{a}} = ({\star} \tau) J \), where \( J \) is the almost complex structure on the surface.

\begin{plem} \label{L:B}
Let \( M \) be an oriented smooth surface embedded in a Weit\-zen\-b{\"o}ck\/ \( 3 \)-manifold. Then
\begin{align}
\Div (HI + 2W^{\mathsf{a}}) & = \overline{\Div} (HI + 2W^{\mathsf{a}}) = \Grad H + J \Grad {\star} \tau, \label{E:B1} \\
\tr(WT)^{\sharp} + W^{\mathsf{t}} (\tr T)^{\sharp} & = H (\tr T)^{\sharp}, \label{E:B2} \\
\tr_M T({{}\cdot{}}, W^{\mathsf{a}} {{}\cdot{}}) & = ({\star} \tau) J (\tr T)^{\sharp}. \label{E:B3}
\end{align}
\end{plem}
\begin{proof}
Since \( \del I = \del J = \bar{\del} I = \bar{\del} J = 0 \), \eqref{E:B1} follows from \cref{L:gradient-divergence}\labelcref{L:gradient-divergence:a}. The left-hand side of~\eqref{E:B2} is equal to \( (\tr S)^{\sharp} \) by \cref{R:S}, but \( S = (\tr W) T = HT \) in dimension two. Lastly, \eqref{E:B3} follows from~\eqref{E:B1} and \cref{L:gradient-divergence}\labelcref{L:gradient-divergence:b}.
\end{proof}

\begin{pthrm} \label{T:B}
Let \( M \) be an oriented smooth surface embedded in a Weit\-zen\-b{\"o}ck\/ \( 3 \)-manifold. Let \( g \colon M \to \mathbb{S}^2 \) be the Gauss map. Let \( \phi \colon TM \isoto g^* T \mathbb{S}^2 \) be the identification in \cref{L:identification}. Then
\begin{align}
- \phi^{-1} \Lap g & = \Grad H + J \Grad {\star} \tau + (HI + ({\star} \tau) J - W) (\tr T)^{\sharp}, \\
- \phi^{-1} \bar{\Lap} g & = \Grad H + J \Grad {\star} \tau + (HI + ({\star} \tau) J) (\tr T)^{\sharp},
\end{align}
where a bar denotes a quantity associated with the Levi-Civita connection on \( M \). In particular, if \( M \) is torsion-free,\/ \( - \phi^{-1} \Lap g = \Grad H + J \Grad {\star} \tau \).
\end{pthrm}
\begin{proof}
\cref{T:A,L:B}.
\end{proof}

\begin{pcor} \label{C:B}
Let \( M \) be an oriented smooth surface embedded in a Weit\-zen\-b{\"o}ck\/ \( 3 \)-manifold. If \( H + \boldsymbol{i} {\star} \tau = 0 \) everywhere, then the Gauss map \( g \colon M \to \mathbb{S}^2 \) is harmonic with respect to the Levi-Civita connection on \( M \).
\end{pcor}

In the torsion-free case, combining with a result of~\cite{MR5011406} yields the following equivalence of three statements.

\begin{pcor} \label{C:C}
Let \( M \) be an oriented torsion-free smooth surface embedded in a Weit\-zen\-b{\"o}ck\/ \( 3 \)-manifold. Then the following three are equivalent.
\begin{enumerate}
\item The Gauss map \( g \colon M \to \mathbb{S}^2 \) is harmonic.
\item \( H + \boldsymbol{i} {\star} \tau \) is a holomorphic function on \( M \).
\item The Hopf differential \( \varphi \) on \( M \) is a holomorphic quadratic differential.
\end{enumerate}
\end{pcor}
\begin{proof}
The equivalence of~(b) and~(c) is just \cref{T:Hopf-differential} (Theorem~B of~\cite{MR5011406}). For the equivalence of~(a) and~(b), let \( z = x + \boldsymbol{i} y \) be a local complex chart on \( M \) such that the metric on \( M \) is of the form \( \lambda^2(z) \, \mathrm{d} z \, \mathrm{d} \bar{z} \) for some positive \( \lambda \). Then \( {\big( \frac{1}{\lambda} \frac{\partial}{\partial x}, \frac{1}{\lambda} \frac{\partial}{\partial y} \big)} \) is a local oriented orthonormal smooth frame for \( M \), so that
\begin{align*}
- \phi^{-1} \Lap g & = \Grad H + J \Grad {\star} \tau \\
& = {\left( \frac{1}{\lambda} \frac{\partial H}{\partial x} \frac{1}{\lambda} \frac{\partial}{\partial x} + \frac{1}{\lambda} \frac{\partial H}{\partial y} \frac{1}{\lambda} \frac{\partial}{\partial y} \right)} + J {\left( \frac{1}{\lambda} \frac{\partial {\star} \tau}{\partial x} \frac{1}{\lambda} \frac{\partial}{\partial x} + \frac{1}{\lambda} \frac{\partial {\star} \tau}{\partial y} \frac{1}{\lambda} \frac{\partial}{\partial y} \right)} \\
& = \frac{1}{\lambda^2} {\left( \frac{\partial H}{\partial x} - \frac{\partial {\star} \tau}{\partial y} \right)} \frac{\partial}{\partial x} + \frac{1}{\lambda^2} {\left( \frac{\partial H}{\partial y} + \frac{\partial {\star} \tau}{\partial x} \right)} \frac{\partial}{\partial y}.
\end{align*}
This vanishes iff \( H + \boldsymbol{i} {\star} \tau \) satisfies the Cauchy-Riemann equations.
\end{proof}

\noindent This directly implies the following classical result.

\begin{pcor} \label{C:Hopf}
Let \( M \) be a connected oriented smooth surface embedded in\/ \( \mathbb{R}^3 \). Then the following three are equivalent.
\begin{enumerate}
\item The Gauss map \( g \colon M \to \mathbb{S}^2 \) is harmonic.
\item \( M \) has constant mean curvature.
\item The Hopf differential \( \varphi \) on \( M \) is a holomorphic quadratic differential.
\end{enumerate}
\end{pcor}

\subsection{Torsion-free hypersurfaces} \label{SS:torsion-free-hypersurfaces}

The significance of torsion-free hypersurfaces in Weit\-zen\-b{\"o}ck geometry is underscored by \cref{T:A,C:A} in \cref{SS:hypersurfaces} and \cref{T:B,C:C} in \cref{SS:surfaces}. Therefore, it is natural to consider whether such hypersurfaces exist. In fact, \cref{T:torsion-free} below guarantees the existence of such hypersurfaces in abundance.

\begin{pthrm} \label{T:torsion-free}
Let \( M \) be a closed orientable smooth hypersurface embedded in a parallelizable smooth manifold \( \tilde{M} \). Then there exists a Weit\-zen\-b{\"o}ck structure on \( \tilde{M} \) such that \( \tilde{T}(X,Y) = 0 \) for all \( X, Y \in \Gamma(TM) \) (in particular, \( M \) is torsion-free).
\end{pthrm}
\begin{proof}
Fix an orientation on \( \tilde{M} \) and a smooth vector field \( N \) along \( M \) that is nowhere tangent to \( M \), determining an orientation on \( M \). Fix any global oriented smooth coframe \( \hat{\varepsilon} \in \Omega^1(\tilde{M},\mathbb{R}^{n+1}) \). Its pullback \( \iota^* \hat{\varepsilon} \in \Omega^1(M,\mathbb{R}^{n+1}) \) by the inclusion \( \iota \colon M \injto \tilde{M} \) yields a \emph{formal immersion} \( \iota^* \hat{\varepsilon} \colon TM \to T \mathbb{R}^{n+1} \) (i.e., a smooth bundle map that is injective on each fiber) covering, e.g., the constant map \( 0 \colon M \to \mathbb{R}^{n+1} \). By the Smale-Hirsch theorem (see Theorem~3.9 of~\cite{MR1225100}, originally due to~\cites{MR0105117,MR0119214}), there exists a smooth immersion \( f \colon M \to \mathbb{R}^{n+1} \) such that both \( \mathrm{d} f \) and \( \iota^* \hat{\varepsilon} \) belong to the same path component of the space of formal immersions from \( TM \) to \( T \mathbb{R}^{n+1} \), equipped with the compact-open topology. Let \( \alpha_{{}\cdot{}} \) be a continuous path of formal immersions \( TM \to T \mathbb{R}^{n+1} \) from \( \alpha_0 = \iota^* \hat{\varepsilon} \) to \( \alpha_1 = \mathrm{d} f \). The map \( \alpha_{{}\cdot{}} \colon TM \times [0,1] \to T \mathbb{R}^{n+1} \) is continuous. Let \( v \colon T \mathbb{R}^{n+1} \surjto \mathbb{R}^{n+1} \) be the projection onto the vector part, so that \( v \alpha_0 = \iota^* \hat{\varepsilon} \) and \( v \alpha_1 = \mathrm{d} f \) as maps \( TM \to \mathbb{R}^{n+1} \).

For each \( p \in M \) and \( t \in [0,1] \), the image of the linear map \( v \alpha_t \colon T_p M \injto \mathbb{R}^{n+1} \) is an oriented hyperplane, determining a unique unit vector \( \nu_t(p) \in \mathbb{R}^{n+1} \) orthogonal to it. If \( (E_1, \dotsc, E_n) \) is a local oriented smooth frame for \( M \), \( \nu_t \) is locally given by
\begin{align}
\nu_t & = \frac{\mathrm{d} \vol_{\mathbb{R}^{n+1}} ({{}\cdot{}}, v \alpha_t E_1, \dotsc, v \alpha_t E_n)^{\sharp}}{|\mathrm{d} \vol_{\mathbb{R}^{n+1}} ({{}\cdot{}}, v \alpha_t E_1, \dotsc, v \alpha_t E_n)^{\sharp}|},
\end{align}
where \( t \in [0,1] \). Clearly, the map \( \nu_t \colon M \to \mathbb{R}^{n+1} \) is smooth for every \( t \in [0,1] \), and the map \( \nu_{{}\cdot{}} \colon M \times [0,1] \to \mathbb{R}^{n+1} \) is continuous. Let \( \tilde{\alpha}_{{}\cdot{}} \colon T \tilde{M} |_M \times [0,1] \to \mathbb{R}^{n+1} \) be the continuous map, linear in the first argument, determined by
\begin{align}
\tilde{\alpha}_t(N) & = \nu_t \quad \text{and} \quad \tilde{\alpha}_t(\iota_* X) = v \alpha_t(X) \quad \text{for all} \ X \in TM,
\end{align}
where \( t \in [0,1] \). For every \( t \in [0,1] \), the map \( \tilde{\alpha}_t \colon T \tilde{M} |_M \to \mathbb{R}^{n+1} \) is smooth, and the map \( \tilde{\alpha}_t \colon T_p \tilde{M} \to \mathbb{R}^{n+1} \) is an orientation-preserving linear bijection for all \( p \in M \). That is, for every \( t \in [0,1] \), \( \tilde{\alpha}_t \) is an oriented smooth coframe for \( \tilde{M} \) along \( M \).

Let \( G_{{}\cdot{}} \colon M \times [0,1] \to \mathrm{GL}^+_{n+1}(\mathbb{R}) \) be the continuous map given by the change-of-basis \( \tilde{\alpha}_t = G_t \tilde{\alpha}_0 \) for each \( t \in [0,1] \). Then \( G_0 = I \), and the map \( G_t \colon M \to \mathrm{GL}^+_{n+1}(\mathbb{R}) \) is smooth for every \( t \in [0,1] \). Using a uniform tubular neighborhood of \( M \) in \( \tilde{M} \), extend the smooth map \( G_1 \colon M \to \mathrm{GL}^+_{n+1}(\mathbb{R}) \) to a smooth map \( G_1 \colon \tilde{M} \to \mathrm{GL}^+_{n+1}(\mathbb{R}) \) (Theorem~5.25 of~\cite{MR3887684}, Corollary~6.27 of~\cite{MR2954043}). Finally, this yields a global oriented smooth coframe \( \tilde{\varepsilon} = G_1 \hat{\varepsilon} \) for \( \tilde{M} \). Its dual \( \tilde{s} = (\tilde{E}_1, \dotsc, \tilde{E}_{n+1}) \) is the desired parallelization for \( \tilde{M} \). Indeed, for all \( X, Y \in \Gamma(TM) \), by Cartan's first equation,
\begin{align}
\tilde{T}(X,Y) & = \sum_{i = 1}^{n+1} \mathrm{d} \tilde{\varepsilon}^i (X,Y) \tilde{E}_i = \sum_{i = 1}^{n+1} \mathrm{d} (G_1 \hat{\varepsilon})^i (X,Y) \tilde{E}_i,
\end{align}
but \( \iota^* G_1 \hat{\varepsilon} = G_1 \iota^* \hat{\varepsilon} = G_1 v \alpha_0 = G_1 \iota^* \tilde{\alpha}_0 = \iota^* \tilde{\alpha}_1 = v \alpha_1 = \mathrm{d} f \) by the construction.
\end{proof}

We now construct an explicit example of an oriented torsion-free smooth surface embedded in a Weit\-zen\-b{\"o}ck \( 3 \)-manifold, using Example~4.4 of~\cite{MR5011406}. Recall that every spatial rotation matrix in \( \mathrm{SO}(3) \) admits an axis-angle representation, called \emph{Rodrigues' rotation formula}, of the form \( \boldsymbol{e}^{\theta \hat{e}} = I + \sin(\theta) \hat{e} + (1 - \cos(\theta)) \hat{e}^2 \in \mathrm{SO}(3) \) for an angle \( \theta \in \mathbb{R} / 2 \boldsymbol{\pi} \mathbb{Z} \) and an axis \( e \in \mathbb{S}^2 \), where \( I \) is the identity matrix and
\begin{align}
\hat{e} & = \begin{pmatrix*}[c]
0 & - e^3 & e^2 \\
e^3 & 0 & - e^1 \\
- e^2 & e^1 & 0
\end{pmatrix*} \in \mathfrak{so}(3).
\end{align}
The linear bijection \( \mathbb{R}^3 \to \mathfrak{so}(3) \;\!;\, e \mapsto \hat{e} \) is called the \emph{hat map}, obtained by the cross product: \( \hat{e} x = e \times x \) for all \( e, x \in \mathbb{R}^3 \). For details, see, e.g., Subsection~4.1 of~\cite{MR5011406}.

Let \( \tilde{M} \) be the Weit\-zen\-b{\"o}ck \( 3 \)-manifold whose underlying smooth manifold is \( \mathbb{R}^3 \) and global smooth frame is \( (\tilde{E}_1, \tilde{E}_2, \tilde{E}_3) = (\partial_1, \partial_2, \partial_3) \cdot G \), where \( (\partial_1, \partial_2, \partial_3) \) is the standard global frame for \( \mathbb{R}^3 \) and \( G = \boldsymbol{e}^{\theta \hat{e}} \colon \mathbb{R}^3 \to \mathrm{SO}(3) \) is the smooth map for a smooth function \( \theta \colon \mathbb{R}^3 \to \mathbb{R} \) and a unit vector \( e \in \mathbb{S}^2 \). Let \( M = \{ z = 0 \} \subseteq \tilde{M} \) be the \( xy \)-plane with \( N = \partial_3 \). Then the torsions \( \tilde{T} \) on \( \tilde{M} \) and \( T \) on \( M \) are given by
\begin{gather}
\tilde{T}(\partial_i, \partial_j) = \theta_i e^k \partial_i + \theta_j e^k \partial_j - (\theta_i e^i + \theta_j e^j) \partial_k, \\
T(\partial_1, \partial_2) = e^3 (\theta_1 \partial_1 + \theta_2 \partial_2), \qquad (\tr T)^{\sharp} = e^3 (\theta_2 \partial_1 - \theta_1 \partial_2)
\end{gather}
for all \( (i,j,k) \in \{ (1,2,3), (2,3,1), (3,1,2) \} \), where \( \theta_i = \partial_i \theta \). Therefore, \( \tilde{M} \) has torsion unless \( \theta \) is constant, but \( M \) is torsion-free iff either \( e^3 = 0 \) or \( \theta |_M \) is constant. The case of nonconstant \( \theta |_M \colon M \to \mathbb{R} \) with \( e^3 = 0 \) is a desired example.

Now, let us compute the Laplacian of the Gauss map \( g = (G^3_1, G^3_2, G^3_3) \colon M \to \mathbb{S}^2 \). Let \( V \in \Gamma(T \mathbb{S}^2) \) be defined by \( V |_p = p \times e \in T_p \mathbb{S}^2 \) for all \( p \in \mathbb{S}^2 \). Then
\begin{align*}
\mathrm{d} g & = \sum_{i = 1}^3 (\mathrm{d} G^3_i \otimes \partial_i) = \sum_{i = 1}^3 ((\mathrm{d} \theta) (G \hat{e})^3_i \otimes \partial_i) = \sum_{i, j = 1}^3 G^3_j \hat{e}^j_i \partial_i \otimes \mathrm{d} \theta \\
& = - \sum_{i, j = 1}^3 \hat{e}^i_j g^j \partial_i \otimes \mathrm{d} \theta = - \sum_{i = 1}^3 (e \times g)^i \partial_i \otimes \mathrm{d} \theta = g^* V \otimes \mathrm{d} \theta,
\end{align*}
where \( g^* V = V \circ g = g \times e \in \Gamma(g^* T \mathbb{S}^2) \). Since \( \mathrm{d} g (\tr T)^{\sharp} = 0 \), the two Laplacians of the Gauss map \( g \colon M \to \mathbb{S}^2 \) coincide in this example by \cref{P:Hessian-Laplacian}. Meanwhile,
\begin{align*}
\sum_{i = 1}^2 \del_{\partial_i} \partial_i & = \sum_{i = 1}^2 \sum_{j = 1}^3 (\tilde{\del}_{\partial_i} (G^i_j \tilde{E}_j))^{\top} = \sum_{i = 1}^2 \sum_{j = 1}^3 (\partial_i G^i_j) \tilde{E}_j^{\top} = \sum_{i, k = 1}^2 \sum_{j = 1}^3 \theta_i (\hat{e} G)^i_j (G^k_j \partial_k) \\
& = \sum_{i, k = 1}^2 \sum_{j, \ell = 1}^3 \theta_i \hat{e}^i_{\ell} G^{\ell}_j G^k_j \partial_k = \sum_{i, k = 1}^2 \theta_i \hat{e}^i_k \partial_k = e^3 (\theta_2 \partial_1 - \theta_1 \partial_2) = (\tr T)^{\sharp}.
\end{align*}
Therefore, the Laplacian of the Gauss map \( g \colon M \to \mathbb{S}^2 \) is
\begin{align*}
\Lap g & = \sum_{i = 1}^2 (\del_{\partial_i} (\mathrm{d} g)) \partial_i = \sum_{i = 1}^2 ((g^* \del^{\mathbb{S}^2})_{\partial_i} (\mathrm{d} g (\partial_i)) - \mathrm{d} g (\del_{\partial_i} \partial_i)) \\
& = \sum_{i = 1}^2 (g^* \del^{\mathbb{S}^2})_{\partial_i} (\theta_i g^* V) = \sum_{i = 1}^2 (\theta_{ii} \, g^* V + \theta_i \del^{\mathbb{S}^2}_{\mathrm{d} g (\partial_i)} V) \\
& = \sum_{i = 1}^2 (\theta_{ii} V |_g + \theta_i^2 (\del^{\mathbb{R}^3}_{V |_g} V)^{\top}) = (\theta_{11} + \theta_{22}) (g \times e) + (\theta_1^2 + \theta_2^2) ((g \times e) \times e)^{\top} \\
& = (\theta_{11} + \theta_{22}) (g \times e) + (\theta_1^2 + \theta_2^2) e^3 (e - e^3 g),
\end{align*}
since \( g \cdot e = (Ge)^3 = e^3 \). Since \( \phi^{-1} e = e \) and \( \phi^{-1} g = N = \partial_3 \), it follows that
\begin{align}
- \phi^{-1} \Lap g & = \underbrace{(\theta_{11} + \theta_{22}) \begin{pmatrix*}[r]
e^2 \\
- e^1
\end{pmatrix*}}_{\textnormal{differential part}} + \underbrace{-(\theta_1^2 + \theta_2^2) e^3 \begin{pmatrix*}[r]
e^1 \\
e^2
\end{pmatrix*}}_{\textnormal{algebraic part}}
\end{align}
relative to the global oriented orthonormal smooth frame \( (\partial_1, \partial_2) \) for \( M \) (cf.~\cref{R:S}). Indeed, using the observation~(c) in Example~4.4 of~\cite{MR5011406}, we obtain
\begin{align}
\Grad H + J \Grad {\star} \tau & = \begin{pmatrix*}[c]
\partial_1 H - \partial_2 {\star} \tau \\
\partial_2 H + \partial_1 {\star} \tau
\end{pmatrix*} = (\theta_{11} + \theta_{22}) \begin{pmatrix*}[r]
e^2 \\ - e^1
\end{pmatrix*}, \\
(HI + ({\star} \tau) J) (\tr T)^{\sharp} & = \begin{pmatrix*}[r]
H & - {\star} \tau \\
{\star} \tau & H
\end{pmatrix*} \begin{pmatrix*}[r]
e^3 \theta_2 \\
- e^3 \theta_1
\end{pmatrix*} = -(\theta_1^2 + \theta_2^2) e^3 \begin{pmatrix*}[r]
e^1 \\
e^2
\end{pmatrix*}
\end{align}
relative to \( (\partial_1, \partial_2) \). The result agrees with \cref{T:B}.

\subsection{Totally skew-symmetric torsion} \label{SS:totally-skew-symmetric-torsion}

In \cref{SS:torsion-free-hypersurfaces}, we considered a condition imposed on hypersurfaces (namely, being torsion-free). In this subsection, we instead consider a condition on the ambient manifold. We begin by investigating a condition under which \emph{every} hypersurface is torsion-free. In fact, it will be observed that nontrivial examples of such ambient manifolds exist only in dimension three, and they are characterized by having \emph{totally skew-symmetric torsion}.

\begin{pdef} \label{D:totally-skew-symmetric}
Let \( V \) be a finite-dimensional real inner product space. A bilinear map \( A \colon V \times V \to V \) is \emph{totally skew-symmetric} iff it is a skew-symmetric bilinear map satisfying one of the following equivalent conditions.
\begin{enumerate}
\item (Orthogonality) \( A(X,Y) \perp X, Y \) for all \( X, Y \in V \).
\item (Alternation) The trilinear map \( A^{\flat} \colon V^3 \to \mathbb{R} \;\!;\, (X,Y,Z) \mapsto \langle X, A(Y,Z) \rangle \) is alternating.
\item (Cyclic symmetry) \( \langle X, A(Y,Z) \rangle = \langle Y, A(Z,X) \rangle \) for all \( X, Y, Z \in V \).
\item (Skew-adjointness) The linear endomorphism \( A(X, {{}\cdot{}}) \colon V \to V \) is skew-adjoint for every \( X \in V \).
\end{enumerate}
\end{pdef}

\begin{plem} \label{L:torsion-free}
Let \( \tilde{M} \) be a Riemann-Cartan manifold. Let \( \tilde{T} \) be the torsion of \( \tilde{M} \). Then the following two are equivalent.
\begin{enumerate}
\item Every smooth hypersurface embedded in \( \tilde{M} \) is torsion-free.
\item At every \( p \in \tilde{M} \), for any hyperplane\/ \( \Pi \subseteq T_p \tilde{M} \), we have \( \tilde{T}(X,Y) \in \Pi^{\perp} \) for all \( X, Y \in \Pi \).
\end{enumerate}
\end{plem}
\begin{proof}
First, suppose~(a). Fix a point \( p \in \tilde{M} \) and a hyperplane \( \Pi \subseteq T_p \tilde{M} \). Let \( \varphi \colon U \to \mathbb{R}^{n+1} \) be a local chart near \( p = \varphi^{-1}(0) \). Take a sufficiently small \( n \)-sphere \( S \) in \( \mathbb{R}^{n+1} \) such that \( 0 \in S \subseteq \varphi[U] \) and \( T_0 S = \varphi_* [\Pi] \). Then \( M = \varphi^{-1}[S] \) is a smooth hypersurface embedded in \( \tilde{M} \) such that \( T_p M = \Pi \). Since \( M \) is torsion-free, \( \tilde{T}(X,Y)^{\top} = 0 \) for all \( X, Y \in T_p M \), from which~(b) follows. Conversely, suppose~(b). Let \( M \) be a smooth hypersurface embedded in \( \tilde{M} \). At every \( p \in M \), we have \( \tilde{T}(X,Y) \in (T_p M)^{\perp} \) for all \( X, Y \in T_p M \). This means that \( \tilde{T}(X,Y)^{\top} = 0 \) for all \( X, Y \in T_p M \). That is, \( M \) is torsion-free.
\end{proof}

\begin{pprop} \label{P:torsion-free}
Let \( \tilde{M} \) be a Riemann-Cartan manifold. Then every smooth hypersurface embedded in \( \tilde{M} \) is torsion-free iff one of the following three holds.
\begin{enumerate}
\item \( \dim \tilde{M} \le 2 \).
\item \( \dim \tilde{M} = 3 \) and the torsion of \( \tilde{M} \) is totally skew-symmetric everywhere.
\item \( \dim \tilde{M} \ge 4 \) and \( \tilde{M} \) is torsion-free.
\end{enumerate}
\end{pprop}
\begin{proof}
The case \( \dim \tilde{M} \le 2 \) is trivial. The case \( \dim \tilde{M} = 3 \) is straightforward by \cref{L:torsion-free}. Consider the case \( \dim \tilde{M} \ge 4 \). The ``if'' direction is trivial. For the ``only if'' direction, it suffices to deduce from~(b) of \cref{L:torsion-free} that \( \tilde{M} \) is torsion-free. Fix \( p \in \tilde{M} \). Let \( B = \{ \tilde{E}_1, \dotsc, \tilde{E}_{n+1} \} \) be an orthonormal basis for \( T_p \tilde{M} \). The claim is that \( \tilde{T}(\tilde{E}_1, \tilde{E}_2) = 0 \). Let \( i \in \{ 1, \dotsc, n+1 \} \). As \( |B| \ge 4 \), one can pick \( \tilde{E}_j \in B \setminus \{ \tilde{E}_1, \tilde{E}_2, \tilde{E}_i \} \). If \( \Pi \subseteq T_p \tilde{M} \) is the hyperplane generated by \( B \setminus \{ \tilde{E}_j \} \), then \( \tilde{T}(\tilde{E}_1, \tilde{E}_2) \in \Pi^{\perp} \) but \( \tilde{E}_i \in \Pi \). Therefore, \( \langle \tilde{E}_i, \tilde{T}(\tilde{E}_1, \tilde{E}_2) \rangle = 0 \).
\end{proof}

The main concern of the rest of this subsection will be totally skew-symmetric torsions. The connection \( \del \) on a Riemann-Cartan manifold \( M \) with totally skew-symmetric torsion \( T \) is simply given by
\begin{align}
\del_X Y & = \bar{\del}_X Y + \tfrac{1}{2} T(X,Y) \qquad (X, Y \in \Gamma(TM)), \label{E:totally-skew-symmetric-torsion}
\end{align}
since the contorsion of \( M \) is \( C = \frac{1}{2} T \) by~\eqref{E:contorsion-torsion}, where \( \bar{\del} \) is the Levi-Civita connection on \( M \). This yields the following observations for hypersurfaces.

\begin{plem} \label{L:totally-skew-symmetric-torsion}
Let \( M \) be a smooth hypersurface embedded in an oriented Riemann-Cartan manifold \( \tilde{M} \) with totally skew-symmetric torsion, oriented by a unit normal vector field \( N \) along \( M \). Then the following five hold, where a bar denotes a quantity associated with the Levi-Civita connection.
\begin{enumerate}
\item \label{L:totally-skew-symmetric-torsion:a} The torsion \( T \) of \( M \) is totally skew-symmetric. The contorsion of \( M \) is\/ \( \frac{1}{2} T \).
\item \label{L:totally-skew-symmetric-torsion:b} For every smooth\/ \( (1,1) \)-tensor field \( A \) on \( M \),\/ \( \tr(AT)^{\sharp} = \tr_M T({{}\cdot{}}, A^{\mathsf{t}} {{}\cdot{}}) \).
\item \label{L:totally-skew-symmetric-torsion:c} The Weingarten map \( W = W^{\mathsf{s}} + W^{\mathsf{a}} \) satisfies \( W^{\mathsf{s}} = \bar{W} \) and \( W^{\mathsf{a}} = \frac{1}{2} \tilde{T}(N, {{}\cdot{}}) \).
\item \label{L:totally-skew-symmetric-torsion:d} The mean curvature satisfies \( H = \bar{H} \).
\item \label{L:totally-skew-symmetric-torsion:e} \( \Div(2W^{\mathsf{a}}) = \overline{\Div}(2W^{\mathsf{a}}) - \tr(W^{\mathsf{a}} T)^{\sharp} \) and\/ \( \overline{\Div}(2W^{\mathsf{a}}) = \tr_M ((\tilde{\del} \tilde{T}) (N, {{}\cdot{}})) \).
\end{enumerate}
\end{plem}
\begin{proof}
Let \( (E_1, \dotsc, E_n) \) be a local oriented orthonormal smooth frame for \( M \). (a) is immediate from \( T = \tilde{T}^{\top} \). For a smooth \( (1,1) \)-tensor field \( A \) on \( M \),
\begin{align*}
\langle E_i, A T(E_j, E_i) \rangle E_j & = \langle A^{\mathsf{t}} E_i, T(E_j, E_i) \rangle E_j = \langle E_j, T(E_i, A^{\mathsf{t}} E_i) \rangle E_j
\end{align*}
for all \( i, j \in \{ 1, \dotsc, n \} \). Summing over \( i, j \in \{ 1, \dotsc, n \} \) yields~(b). \eqref{E:totally-skew-symmetric-torsion} implies~(c), and~(c) implies~(d). Since \( (\tr T)^{\sharp} = 0 \) and \( \tr_M T({{}\cdot{}}, W^{\mathsf{a}} {{}\cdot{}}) = - \tr(W^{\mathsf{a}} T)^{\sharp} \) by~(b), the former equality of~(e) follows from \cref{L:gradient-divergence}\labelcref{L:gradient-divergence:b}. By~(c) and~\eqref{E:tau2},
\begin{align*}
(\del_{E_i} 2W^{\mathsf{a}}) E_i & = \del_{E_i} (2W^{\mathsf{a}} E_i) - 2W^{\mathsf{a}} \del_{E_i} E_i = \del_{E_i} (\tilde{T}(N, E_i)) - \tilde{T}(N, \del_{E_i} E_i) \\
& = \tilde{\del}_{E_i} (\tilde{T}(N, E_i)) - \II(E_i, \tilde{T}(N, E_i)) N - \tilde{T}(N, \tilde{\del}_{E_i} E_i) \\
& = (\tilde{\del}_{E_i} \tilde{T}) (N, E_i) + \tilde{T}(\tilde{\del}_{E_i} N, E_i) - \langle WE_i, 2W^{\mathsf{a}} E_i \rangle N \\
& = (\tilde{\del}_{E_i} \tilde{T}) (N, E_i) + \tilde{T}(E_i, WE_i) - \tau(E_i, WE_i) N \\
& = (\tilde{\del}_{E_i} \tilde{T}) (N, E_i) + T(E_i, WE_i)
\end{align*}
for all \( i \in \{ 1, \dotsc, n \} \). Summing over \( i \in \{ 1, \dotsc, n \} \) yields
\begin{align*}
\Div(2W^{\mathsf{a}}) & = \tr_M ((\tilde{\del} \tilde{T}) (N, {{}\cdot{}})) + \tr_M T({{}\cdot{}}, W {{}\cdot{}}) = \tr_M ((\tilde{\del} \tilde{T}) (N, {{}\cdot{}})) + \tr(W^{\mathsf{t}} T)^{\sharp}
\end{align*}
by~(b), but \( \tr(W^{\mathsf{s}} T)^{\sharp} = 0 \) by \cref{L:contraction}\labelcref{L:contraction:b}. The latter equality of~(e) follows.
\end{proof}

Consequently, we obtain the following result.

\begin{pthrm} \label{T:D}
Let \( M \) be an oriented smooth hypersurface embedded in a Weit\-zen\-b{\"o}ck\/ \( (n+1) \)-manifold with parallel totally skew-symmetric torsion. Let \( g \colon M \to \mathbb{S}^n \) be the Gauss map, and let \( \hat{g} = \hat{\iota} \circ g \colon M \to \mathbb{R}^{n+1} \), where \( \hat{\iota} \colon \mathbb{S}^n \injto \mathbb{R}^{n+1} \) is the inclusion. Let \( \phi \colon TM \isoto g^* T \mathbb{S}^n \) be the identification in \cref{L:identification}. Then
\begin{align}
- \phi^{-1} \Lap g & = \Grad H, \qquad - \phi^{-1} \Lap \hat{g} = \Grad H + \| W \|^2 N.
\end{align}
Also,\/ \( \Lap g = \bar{\Lap} g \),\/ \( \Lap \hat{g} = \bar{\Lap} \hat{g} \), \( H = \bar{H} \), and\/ \( \| W \|^2 = \| \bar{W} \|^2 + \widetilde{\Ric}(N,N) \), where a bar denotes a quantity associated with the Levi-Civita connection and\/ \( \widetilde{\Ric} \) denotes the Ricci curvature of the ambient Levi-Civita connection.
\end{pthrm}
\begin{proof}
By \cref{L:totally-skew-symmetric-torsion}\labelcref{L:totally-skew-symmetric-torsion:b}, \( (\tr T)^{\sharp} = 0 \) and \( \tr_M T({{}\cdot{}}, W^{\mathsf{a}} {{}\cdot{}}) = - \tr(W^{\mathsf{a}} T)^{\sharp} \). Also, by \cref{L:totally-skew-symmetric-torsion}\labelcref{L:totally-skew-symmetric-torsion:a} and \cref{L:contraction}\labelcref{L:contraction:b}, \( \tr(W^{\mathsf{s}} T)^{\sharp} = 0 \). These three equalities imply that \( - \phi^{-1} \Lap g = \overline{\Div} (HI + 2W^{\mathsf{a}}) \) by \cref{T:A}, but \( \overline{\Div}(2W^{\mathsf{a}}) = 0 \) by \cref{L:totally-skew-symmetric-torsion}\labelcref{L:totally-skew-symmetric-torsion:e}. The first equality follows, and hence the second equality follows from \cref{T:Laplacian}. Now, \( \Lap g = \bar{\Lap} g \) and \( \Lap \hat{g} = \bar{\Lap} \hat{g} \) by \cref{P:Hessian-Laplacian}, and \( H = \bar{H} \) by \cref{L:totally-skew-symmetric-torsion}\labelcref{L:totally-skew-symmetric-torsion:d}. Since \( \| W \|^2 = \| W^{\mathsf{s}} \|^2 + \| W^{\mathsf{a}} \|^2 \), it remains to show that \( \widetilde{\Ric}(N,N) = \tfrac{1}{4} \| \tilde{T}(N, {{}\cdot{}}) \|^2 \) by \cref{L:totally-skew-symmetric-torsion}\labelcref{L:totally-skew-symmetric-torsion:c}, where \( \tilde{T} \) is the ambient torsion. Using the ambient contorsion \( \frac{1}{2} \tilde{T} \), one can deduce from~\eqref{E:curvature-contorsion} that the (Levi-Civita) ambient curvature \( \tilde{R} \) satisfies \( \tilde{R}(X,N) N = \frac{1}{4} \tilde{T}(N, \tilde{T}(X,N)) \) for all \( X \in \Gamma(TM) \). Then \( \widetilde{\tr} (\tilde{R}({{}\cdot{}}, N) N) \) is equal to
\begin{align*}
\tr (\tilde{R}({{}\cdot{}}, N) N) & = \tfrac{1}{4} \tr_M \langle {{}\cdot{}}, \tilde{T}(N, \tilde{T}({{}\cdot{}}, N)) \rangle = \tfrac{1}{4} \tr_M \langle \tilde{T}(N, {{}\cdot{}}), \tilde{T}(N, {{}\cdot{}}) \rangle,
\end{align*}
where \( \widetilde{\tr} \) denotes the ambient trace.
\end{proof}

This recovers the result of~\cite{MR2002821} about ambient Lie groups as follows.

\begin{pcor}[Theorem~1 of~\cite{MR2002821}] \label{C:Lie-group}
Let \( M \) be a smooth hypersurface embedded in an oriented Lie group \( G \) with a bi-invariant metric, oriented by a unit normal vector field \( N \) along \( M \). Then the Laplacian of the smooth map \( \bar{g} \colon M \to \mathfrak{g} \) defined by \( \bar{g}(p) = \mathrm{d} L_p^{-1} (N |_p) \) is given by
\begin{align}
- \Lap \bar{g} |_p & = \mathrm{d} L_p^{-1} (\Grad \bar{H}) + {\big( \| \bar{W} \|^2 + \widetilde{\Ric}(N,N) \big)} \, \bar{g}(p) \qquad (p \in M).
\end{align}
\end{pcor}
\begin{proof}
Let \( (\tilde{E}_1, \dotsc, \tilde{E}_{n+1}) \) be a left-invariant global oriented orthonormal smooth frame for \( G \). Then it defines a Weit\-zen\-b{\"o}ck structure on \( G \) with parallel totally skew-symmetric torsion. By \cref{T:D},
\begin{align}
- \phi^{-1} \Lap \hat{g} & = \Grad \bar{H} + {\big( \| \bar{W} \|^2 + \widetilde{\Ric}(N,N) \big)} N,
\end{align}
where \( \hat{g} \colon M \to \mathbb{R}^{n+1} \). Fix \( p \in M \). Since \( \bar{g} = \sum_{i = 1}^{n+1} \hat{g}^i \tilde{E}_i |_1 \) on \( M \), applying \( \mathrm{d} L_p^{-1} \) yields \( \mathrm{d} L_p^{-1} \phi^{-1} \Lap \hat{g} |_p = \Lap \bar{g} |_p \in \mathfrak{g} \) and \( \mathrm{d} L_p^{-1} (N |_p) = \bar{g}(p) \in \mathfrak{g} \).
\end{proof}

Now, consider surfaces. For an oriented Riemannian \( 3 \)-manifold, every totally skew-symmetric smooth \( (1,2) \)-tensor field is proportional to the \emph{cross product} by \cref{L:cross-product} below, which is the totally skew-symmetric smooth \( (1,2) \)-tensor field defined by \( X \times Y = \mathrm{d} \vol ({{}\cdot{}}, X, Y)^{\sharp} \) for all smooth vector fields \( X \) and \( Y \).

\begin{plem} \label{L:cross-product}
Let \( A \) be a smooth\/ \( (1,2) \)-tensor field on an oriented Riemannian\/ \( 3 \)-manifold \( M \). Then \( A \) is totally skew-symmetric iff there exists a smooth function \( f \colon M \to \mathbb{R} \) such that \( A(X,Y) = f X \times Y \) for all \( X, Y \in \Gamma(TM) \).
\end{plem}
\begin{proof}
If \( A \) is totally skew-symmetric, then \( A^{\flat} \in \Omega^3(M,\mathbb{R}) \) (\cref{D:totally-skew-symmetric}), and thus there is a smooth function \( f \colon M \to \mathbb{R} \) such that \( A^{\flat} = f \, \mathrm{d} \vol \).
\end{proof}

For an oriented Riemann-Cartan \( 3 \)-manifold with totally skew-symmetric torsion, the curvature is given as follows.

\begin{pprop} \label{P:curvature}
Let \( M \) be an oriented Riemann-Cartan\/ \( 3 \)-manifold with totally skew-symmetric torsion \( f {\times} \). For all \( X, Y, Z \in \Gamma(TM) \),
\begin{align}
R(X,Y) Z & = \bar{R}(X,Y) Z + \tfrac{1}{4} f^2 (X \times Y) \times Z + \tfrac{1}{2} (X(f) Y - Y(f) X) \times Z, \\
\Ric(X,Y) & = \overline{\Ric}(X,Y) - \tfrac{1}{2} f^2 \langle X, Y \rangle + \tfrac{1}{2} \, \mathrm{d} f (X \times Y).
\end{align}
\end{pprop}
\begin{proof}
Fix \( X, Y, Z \in \Gamma(TM) \). We substitute the contorsion in~\eqref{E:curvature-contorsion} with \( C = \frac{1}{2} f {\times} \). Since the cross product is parallel, \( (\bar{\del}_X C) (Y,Z) = \tfrac{1}{2} X(f) Y \times Z \) and likewise with \( X \) and \( Y \) interchanged. Also, by the Jacobi identity, \( C(X, C(Y,Z)) - C(Y, C(X,Z)) \) is equal to \( \tfrac{1}{4} f^2 (X \times Y) \times Z \). The former equality follows. Taking the trace then yields the latter equality.
\end{proof}

\begin{pthrm} \label{T:Schur-lemma}
Let \( M \) be an oriented Riemann-Cartan\/ \( 3 \)-manifold with totally skew-symmetric torsion \( f {\times} \). Then the following three are equivalent.
\begin{enumerate}
\item The torsion of \( M \) is parallel.
\item The Ricci curvature of \( M \) is symmetric.
\item The smooth function \( f \colon M \to \mathbb{R} \) is locally constant.
\end{enumerate}
\end{pthrm}
\begin{proof}
Since the cross product is parallel, the torsion is parallel iff \( \mathrm{d} f (X) Y \times Z = 0 \) for all \( X, Y, Z \in \Gamma(TM) \) iff \( \mathrm{d} f = 0 \). By \cref{P:curvature}, the Ricci curvature is symmetric iff \( \mathrm{d} f (X \times Y) = 0 \) for all \( X, Y \in \Gamma(TM) \) iff \( \mathrm{d} f = 0 \).
\end{proof}

This yields several corollaries.

\begin{pcor}[Schur's lemma for Riemann-Cartan \( 3 \)-manifolds] \label{C:Schur-lemma}
Let \( M \) be an oriented Riemann-Cartan\/ \( 3 \)-manifold. Suppose that there exist two smooth functions \( f_1, f_2 \colon M \to \mathbb{R} \) such that
\begin{align}
\Ric(X,Y) & = f_1 \langle X, Y \rangle, \qquad T(X,Y) = f_2 X \times Y \qquad (X, Y \in \Gamma(TM)). \label{E:Schur-lemma}
\end{align}
Then both \( f_1 \) and \( f_2 \) must be constant on each connected component of \( M \).
\end{pcor}
\begin{proof}
By \cref{T:Schur-lemma}, \( f_2 \) is locally constant. Then, by \cref{P:curvature}, the Ricci curvature of the Levi-Civita connection is given by \( \overline{\Ric} = {\left( f_1 + \tfrac{1}{2} f_2^2 \right)} \langle {{}\cdot{}}, {{}\cdot{}} \rangle \). By Schur's lemma for Riemannian manifolds, \( f_1 + \frac{1}{2} f_2^2 \) (hence \( f_1 \)) is locally constant.
\end{proof}

\noindent It is worth mentioning that the condition~\eqref{E:Schur-lemma} in \cref{C:Schur-lemma} also appears in Proposition~4.18 of~\cite{MR5011406}.

\begin{pcor} \label{C:tau}
Let \( \tilde{M} \) be a Weit\-zen\-b{\"o}ck\/ \( 3 \)-manifold with totally skew-symmetric torsion \( \tilde{T} \). Then \( \tilde{T} = \lambda {\times} \) for some locally constant function \( \lambda \colon \tilde{M} \to \mathbb{R} \). Moreover, every oriented smooth surface embedded in \( \tilde{M} \) is torsion-free and satisfies\/ \( {\star} \tau = \lambda \).
\end{pcor}

\noindent A consequence of \cref{C:tau} is that a connected Weit\-zen\-b{\"o}ck \( 3 \)-manifold with nonzero totally skew-symmetric torsion does not admit an oriented smooth surface satisfying \( H + \boldsymbol{i} {\star} \tau = 0 \), which is minimal in the sense of~\cite{MR5011406}.

\begin{pcor}
Let \( M \) be a connected oriented smooth surface embedded in a Weit\-zen\-b{\"o}ck\/ \( 3 \)-manifold with totally skew-symmetric torsion. Let \( g \colon M \to \mathbb{S}^2 \) be the Gauss map. Let \( \phi \colon TM \isoto g^* T \mathbb{S}^2 \) be the identification in \cref{L:identification}. Then
\begin{align}
- \phi^{-1} \Lap g & = \Grad H.
\end{align}
Moreover, the following three are equivalent.
\begin{enumerate}
\item The Gauss map \( g \colon M \to \mathbb{S}^2 \) is harmonic.
\item \( H \) (i.e., \( \bar{H} \)) is constant.
\item The Hopf differential \( \varphi \) on \( M \) is a holomorphic quadratic differential.
\end{enumerate}
\end{pcor}

Not every parallelizable smooth manifold admits a Weit\-zen\-b{\"o}ck structure with totally skew-symmetric torsion (cf.~\cref{T:torsion-free}). Indeed, the classification of flat Riemann-Cartan manifolds with totally skew-symmetric torsion was established by Cartan and Schouten, Wolf~\cites{MR0312442,MR0312443}, and Agricola and Friedrich~\cite{MR2651537}. Due to the classification, up to universal cover, a connected complete flat Riemann-Cartan manifold with totally skew-symmetric torsion must be the product of Lie groups or the \( 7 \)-sphere \( \mathbb{S}^7 \).

\subsection{HOS' theorem} \label{SS:HOS-theorem}

One of the applications of Ruh and Vilms' theorem is a classification result of Hoffman, Osserman, and Schoen~\cite{MR0694604}. Corresponding analogues have also been studied in many works extending Ruh and Vilms' theorem. In this subsection, we establish two such analogues in the Weit\-zen\-b{\"o}ck setting, which generalize the Euclidean case.

\begin{pthrm} \label{T:HOS1}
Let \( M \) be a connected closed oriented torsion-free smooth hypersurface embedded in a Weit\-zen\-b{\"o}ck\/ \( (n+1) \)-manifold, such that the smooth\/ \( (1,1) \)-tensor field \( HI + 2W^{\mathsf{a}} \) is divergence-free. Then the following two are equivalent.
\begin{enumerate}
\item The image of the Gauss map \( g \colon M \to \mathbb{S}^n \) is contained in a closed hemisphere.
\item The image of the Gauss map \( g \colon M \to \mathbb{S}^n \) is contained in an equator.
\end{enumerate}
\end{pthrm}
\begin{proof}
Suppose that there exists \( V \in \mathbb{S}^n \) such that \( \langle V, g(p) \rangle \le 0 \) for all \( p \in M \). Let \( \hat{g} = \hat{\iota} \circ g \colon M \to \mathbb{R}^{n+1} \), where \( \hat{\iota} \colon \mathbb{S}^n \injto \mathbb{R}^{n+1} \) is the inclusion. By \cref{T:A},
\begin{align}
\Lap \langle V, \hat{g} \rangle & = \langle V, \Lap \hat{g} \rangle = \langle V, - \| W \|^2 \phi N \rangle = - \| W \|^2 \langle V, g \rangle \ge 0,
\end{align}
so that \( \langle V, \hat{g} \rangle \colon M \to \mathbb{R} \) is subharmonic. Since \( M \) is closed, it must be constant. It follows that either \( \langle V, g \rangle = 0 \) everywhere or \( W = 0 \) everywhere. In the former case, the image \( g[M] \) is contained in the equator \( {\left\{ x \in \mathbb{S}^n \, \middle| \, \langle V, x \rangle = 0 \right\}} \). In the latter case, the Gauss map is constant by \cref{P:identification}.
\end{proof}

\begin{pthrm} \label{T:HOS2}
Let \( M \) be a connected complete oriented torsion-free smooth surface embedded in a Weit\-zen\-b{\"o}ck\/ \( 3 \)-manifold, such that \( H + \boldsymbol{i} {\star} \tau \) is a holomorphic function on \( M \). Then the following two are equivalent.
\begin{enumerate}
\item The image of the Gauss map \( g \colon M \to \mathbb{S}^2 \) is contained in a closed hemisphere.
\item The image of the Gauss map \( g \colon M \to \mathbb{S}^2 \) is contained in an equator.
\end{enumerate}
Furthermore, if either holds, \( M \) must be a flat Riemannian surface, and hence isometric to either the plane\/ \( \mathbb{R}^2 \), a cylinder\/ \( \mathbb{R}^2 / a \mathbb{Z} \) for some nonzero \( a \in \mathbb{R}^2 \), or a torus\/ \( \mathbb{R}^2 / (a \mathbb{Z} + b \mathbb{Z}) \) for some basis\/ \( \{ a, b \} \subseteq \mathbb{R}^2 \).
\end{pthrm}
\begin{proof}
Suppose that there exists \( V \in \mathbb{S}^2 \) such that \( \langle V, g(p) \rangle \le 0 \) for all \( p \in M \). Let \( \hat{g} = \hat{\iota} \circ g \colon M \to \mathbb{R}^3 \), where \( \hat{\iota} \colon \mathbb{S}^2 \injto \mathbb{R}^3 \) is the inclusion. Consider the lift \( \langle V, \hat{g} \rangle \colon \hat{M} \to \mathbb{R} \) to the universal cover \( \hat{M} \) of \( M \). By \cref{T:B},
\begin{align}
\Lap \langle V, \hat{g} \rangle & = \langle V, \Lap \hat{g} \rangle = \langle V, - \| W \|^2 \phi N \rangle = - \| W \|^2 \langle V, g \rangle \ge 0,
\end{align}
so that \( \langle V, \hat{g} \rangle \colon \hat{M} \to \mathbb{R} \) is bounded and subharmonic. By the uniformization theorem, \( \hat{M} \) is conformally equivalent to either the sphere, the plane, or the unit disc.

In the case of the sphere or the plane, \( \langle V, \hat{g} \rangle \colon \hat{M} \to \mathbb{R} \) must be constant. It follows that either \( \langle V, g \rangle = 0 \) everywhere or \( W = 0 \) everywhere. In the former case, the image \( g[M] \) is contained in the equator \( {\left\{ x \in \mathbb{S}^2 \, \middle| \, \langle V, x \rangle = 0 \right\}} \). In the latter case, the Gauss map is constant by \cref{P:identification}. In either case, (b) follows. Consider the case of the unit disc. Since \( \det W \) is equal to the intrinsic Gaussian curvature \( K \) of \( M \) (Proposition~3.8 of~\cite{MR5011406}), a direct computation yields \( \| W \|^2 = |H + \boldsymbol{i} {\star} \tau|^2 - 2K \), and thus the nonpositive smooth function \( f = \langle V, \hat{g} \rangle \colon \hat{M} \to \mathbb{R} \) satisfies
\begin{align}
\Lap f - 2K f + |H + \boldsymbol{i} {\star} \tau|^2 f & = 0.
\end{align}
By Corollary~3 of~\cite{MR0562550}, \( f = 0 \) at some point in \( \hat{M} \), and hence \( f = 0 \) everywhere. Then the image \( g[M] \) is contained in the equator \( {\left\{ x \in \mathbb{S}^2 \, \middle| \, \langle V, x \rangle = 0 \right\}} \). (b) follows.

Now, to show that \( M \) is flat, suppose~(b). Fix \( V \in \mathbb{S}^2 \) such that \( \langle V, g(p) \rangle = 0 \) for all \( p \in M \). Let \( E = V^1 \tilde{E}_1 + V^2 \tilde{E}_2 + V^3 \tilde{E}_3 \) be the unit smooth vector field, where \( (\tilde{E}_1, \tilde{E}_2, \tilde{E}_3) \) is the ambient global smooth frame. Since \( \langle N, E \rangle = \langle g, V \rangle = 0 \), it is tangent to \( M \). Then \( (E, JE) \) is a parallel global oriented orthonormal smooth frame for \( M \), where \( J \) is the almost complex structure on \( M \). This implies that the curvature of \( M \) vanishes. The classification of a connected complete flat Riemannian surface is standard. See, e.g., Section~2.5 of~\cite{MR2742530}.
\end{proof}

\begin{pcor}[Theorem~1 of~\cite{MR0694604}] \label{C:HOS}
Let \( M \) be a connected complete oriented CMC smooth surface embedded in\/ \( \mathbb{R}^3 \). If the image of the Gauss map \( g \colon M \to \mathbb{S}^2 \) is contained in a closed hemisphere, then \( M \) is either a plane or a cylinder.
\end{pcor}

\subsubsection*{Acknowledgements}

The author has been supported by a KIAS Individual Grant (MG108901) at Korea Institute for Advanced Study.

\end{document}